\documentclass[review,onefignum,onetabnum]{siamart171218}

\usepackage{lipsum}
\usepackage{mathrsfs}
\usepackage{amsfonts}
\usepackage{graphicx}
\usepackage{epstopdf}
\usepackage{algorithmic}
\usepackage{amssymb}
\usepackage{amsmath}

\ifpdf
  \DeclareGraphicsExtensions{.eps,.pdf,.png,.jpg}
\else
  \DeclareGraphicsExtensions{.eps}
\fi

\newsiamremark{hypothesis}{Hypothesis}
\crefname{hypothesis}{Hypothesis}{Hypotheses}
\newsiamthm{claim}{Claim}

\nolinenumbers

\headers{Gegenbauer Polynomials}{T. M. Dunster}

\title{Uniform asymptotic expansions for Gegenbauer polynomials and related functions via differential equations having a simple pole}

\author{T. M. Dunster\thanks{Department of Mathematics and Statistics, San Diego State University, 5500 Campanile Drive, San Diego, CA 92182-7720, USA. 
  (\email{mdunster@sdsu.edu}, \url{https://tmdunster.sdsu.edu}).}}

\usepackage{amsopn}

\makeatletter
\newcommand*{\addFileDependency}[1]{
  \typeout{(#1)}
  \@addtofilelist{#1}
  \IfFileExists{#1}{}{\typeout{No file #1.}}
}
\makeatother

\nolinenumbers

\ifpdf
\hypersetup{
  pdftitle={Uniform asymptotic expansions for Gegenbauer polynomials and related functions via differential equations having a simple pole},
  pdfauthor={T. M. Dunster}
}

\begin{document}

\maketitle

\begin{abstract}
  Asymptotic expansions are derived for Gegenbauer (ultraspherical) polynomials for large order $n$ that are uniformly valid for unbounded complex values of the argument $z$, including the real interval $0 \leq z \leq 1$ in which the zeros in the right half plane are located: symmetry extends the results to the left half plane. The approximations are derived from the differential equation satisfied by these polynomials, and other independent solutions are also considered. For large $n$ this equation is characterized by having a simple pole, and expansions valid at this singularity involve Bessel functions and slowly varying coefficient functions. The expansions for these functions are simpler than previous approximations, in particular being computable to a high degree of accuracy. Simple explicit error bounds are derived which only involve elementary functions, and thereby provide a simplification of previous expansions and error bounds associated with differential equations having a large parameter and simple pole.

\end{abstract}



\section{Introduction}
\label{sec1}

From \cite[Eq. (4.7.31)]{Szego:1975:OP} Gegenbauer polynomials, also known as ultraspherical polynomials, have the explicit representation
\begin{equation}
\label{eq1.1}
C^{(\lambda)}_{n}(z)=
\sum_{k=0}^{\left \lfloor n/2\right \rfloor}(-1)^{k}
\frac{\Gamma (n-k+\lambda )}
{\Gamma (\lambda )k!(n-2k)!}(2z)^{n-2k}
\quad (n=0,1,2,\ldots),
\end{equation}
and satisfy the second order differential equation
\begin{equation}
\label{eq1.2}
(1-z^2)y'' - (2\lambda + 1)zy' + n(n + 2\lambda)y=0.
\end{equation}

An alternative representation is given in terms of the hypergeometric function, namely \cite[Eqs. 15.9.2, 18.5.10]{NIST:DLMF}
\begin{multline}
\label{eq1.3}
C^{(\lambda)}_{n}(z)=\frac{{(2\lambda)_{n}}}{n!}
F\left(-n,2\lambda+n;\lambda+\tfrac{1}{2};
\tfrac{1}{2}(1-z)\right)  \\
=\frac{(\lambda)_{n}(2z)^{n}}{n!}
F\left(-\frac{1}{2}n,-\frac{1}{2}n+\tfrac{1}{2};
1-\lambda-n;\frac{1}{z^2}\right).
\end{multline}

As $z \rightarrow 1$ we have from (\ref{eq1.3}), on noting from its definition \cite[Eq. 15.2.1]{NIST:DLMF} that $F(a,b;c,z)=1+\mathcal{O}(z)$ as $z \rightarrow 0$,
\begin{equation}
\label{eq1.6}
C^{(\lambda)}_{n}(z)=\frac{{(2\lambda)_{n}}}{n!}
\left\{1+\mathcal{O}(z-1) \right\}.
\end{equation}

We also note from (\ref{eq1.1}) (or (\ref{eq1.3})) that
\begin{equation}
\label{eq1.7}
C^{(\lambda)}_{n}(z)
=\frac{(\lambda)_{n}(2z)^{n}}{n!}
\left\{1+\mathcal{O}\left(\frac{1}{z}\right)\right\}
\quad (z \rightarrow \infty).
\end{equation}

A solution that is $\mathcal{O}(z^{-n-2\lambda})$ as $z \to \infty$ for $|\arg(z)|<\pi$, will be defined below (see (\ref{eq1.5}) and (\ref{eq1.15})). Thus this is the recessive solution of (\ref{eq1.2}) at that singularity, whereas $C^{(\lambda)}_{n}(z)$ is a dominant solution.

In this paper we shall derive uniform asymptotic expansions for these functions for large positive integer values of $n$, and the differential equation (\ref{eq1.2}) plays a central role in deriving these. Throughout this paper we assume that $\lambda$ is nonnegative and bounded, with $z$ taking unbounded real or complex values in the right half plane $\Re(z) \geq 0$. Extension to the left half plane follows immediately from
\begin{equation}
\label{eq2.22b}
C^{(\lambda)}_{n}(-z)=
(-1)^n C^{(\lambda)}_{n}(z),
\end{equation}
which follows from (\ref{eq1.1}).

Gegenbauer polynomials have many applications, including potential theory and harmonic analysis. They are also important on account of their relationship to the Chebyshev polynomials $T_{n}(z)$, since these and their derivatives play a central role in numerical analysis. The asymptotic behavior of $T_{n}(z)$ for large $n$ follows trivially from their explicit representation $T_{n}(\cos(\theta))=\cos(n\theta)$, but in contrast the behavior of their derivatives is not so straightforward. However, for bounded $r$th derivatives, uniform asymptotic expansions as $n \rightarrow \infty$ follow immediately from the results in this paper, via the following simple identity (which is somewhat difficult to find in the literature).

\begin{lemma}
For $r=1,2,3,\ldots ,n$
\begin{equation}
\label{eq1.8}
\frac{d^{r}}{dz^{r}}T_{n}(z)
=2^{r-1} (r-1)! n C^{(r)}_{n-r}(z).
\end{equation}
\end{lemma}

\begin{proof}

From \cite[Eqs. 15.5.2, 18.5.7, 18.7.3]{NIST:DLMF}
\begin{equation*}
\label{eq1.9}
T_{n}(z)=F\left ( -n,n;\tfrac{1}{2};\tfrac{1}{2}(1-z) \right ),
\end{equation*}
and 
\begin{equation*}
\label{eq1.10}
\frac{d^{r}}{dz^{r}}F(a,b;c;z)=
\frac{(a)_{r}(b)_{r}}{(c)_{r}}F(a+r,b+r;c+r;z).
\end{equation*}
Then use these along with (\ref{eq1.3}) and (\ref{eqGAMMA}) below.
\end{proof}

Gegenbauer polynomials are well known to be special cases of the Jacobi polynomials, and (see for example \cite[Eq. 18.7.1]{NIST:DLMF}) this relationship is given by
\begin{equation}
\label{eq1.4}
C^{(\lambda)}_{n}(z)
=\frac{(2\nu+1)_{n}}
{(\nu+1)_{n}}
P_{n}^{(\nu,\nu)}(z),
\end{equation}
where for convenience we have introduced  the parameter
\begin{equation}
\label{eq1.5}
\nu = \lambda -\tfrac{1}{2}.
\end{equation}

Most of the existing asymptotic results for $C^{(\lambda)}_{n}(z)$ follow from those for Jacobi polynomials and then appealing to (\ref{eq1.4}). In \cite{Baratella:1988:BET} an asymptotic expansion was provided for the Jacobi polynomial $P_{n}^{(p,q)}(z)$ in terms of Bessel functions, holding uniformly for $-1 <-1+\delta \leq z \leq 1$. This used the theory of differential equations having a simple pole and a large parameter, given by Olver \cite[Chap. 12]{Olver:1997:ASF}; see (\ref{eq2.3}) - (\ref{eq2.5}) below. In \cite{Elliott:1971:AEJ} results for complex $z$ were obtained, also using Olver's method and valid at $z=1$, with asymptotic expansions furnished for the Jacobi polynomial $P_{n}^{(p,q)}(z)$ and a companion solution $Q_{n}^{(p,q)}(z)$ (see (\ref{eq1.11}) below). There are no error bounds and the coefficients are also hard to compute.

In \cite{Boyd:1990:CSP} and \cite{Dunster:2004:CSP} the authors also considered differential equations with a simple pole, but took a different approach to Olver. In both papers exact, rather than asymptotic, solutions were considered, which again involved Bessel functions, but two slowly varying coefficient functions instead of two asymptotic series. These are analytic at the pole $z=1$, and in both cases the results were applied to associated Legendre functions. In \cite{Boyd:1990:CSP} it was shown how these functions could be expanded as a traditional asymptotic series, and error bounds were derived for both. The coefficients are as difficult to compute as Olver's (indeed are equivalent), but while the bounds are simpler than Olver's, they still involve the coefficients and hence are not easy to compute beyond a few terms. In \cite{Dunster:2004:CSP} the expansions for the slowly-varying coefficient functions were convergent in certain unbounded regions of the complex plane containing the pole. Being convergent no error bounds were necessary, but the coefficients involved are again hard to compute.

Returning to known results for Jacobi and related functions, the most comprehensive asymptotic results were given in \cite{Dunster:1999:AEJ}. Using a differential equation approach for a coalescing turning point and a simple or double pole, uniform asymptotic approximations and expansions were derived, with explicit error bounds, that are uniformly valid in unbounded complex domains containing the singularity, with a wide range of parameters allowed (in particular for small to large $\lambda$ when specialized to Gegenbauer polynomials). The price paid is that the transformations of variables, coefficients and error bounds are quite complicated.

Integral methods have also been commonly used for the asymptotic approximation of Jacobi polynomials; for example in \cite{Frenzen:1985:AJP} an expansion for $P_{n}^{(p,q)}(z)$ was obtained in terms of a sequence of Bessel functions of increasing order, uniformly valid in the same interval as in \cite{Baratella:1988:BET}. Also using an integral method, in \cite{Bai:2007:UEJ} and \cite{Wong:2004:JCD} expansions similar to that given in \cite{Baratella:1988:BET} were derived, and in \cite{Ursell:2007:INC} and \cite{Wong:2003:EET} the case of $z$ complex lying in a neighborhood of the singularity at $z=1$ was considered.

Our approach is similar to \cite{Boyd:1990:CSP} and \cite{Dunster:2004:CSP} in that we also consider exact coefficient functions for solutions, but we differ from those two papers by following the method of \cite{Dunster:2017:COA} for the case of a simple turning point. It centers on the Liouville-Green (LG) expansions of exponential form along with the error bounds constructed in \cite{Dunster:2020:LGE}. We obtain expansions involving coefficients that are straightforward to evaluate recursively, and likewise have explicit and readily computable error bounds, at least for $z$ not too close to the pole. Moreover, as in \cite{Dunster:2017:COA} and \cite{Dunster:2021:SEB}, we use Cauchy's integral formula to evaluate the approximations in a full neighborhood of the pole, as well as bound the error terms there.

The significance of the present paper is that our expansions are uniformly valid in unbounded complex domains, including the oscillatory interval, with coefficients and error bounds that are not only simpler and easier to compute for the Gegenbauer polynomials considered here, but also when applied to general asymptotic solutions of differential equations having a simple pole.

Let us now introduce some numerically satisfactory solutions of (\ref{eq1.2}). These are important in their own right, and also play an important role when constructing expansions valid at the pole. For the first of these we consider the Jacobi function defined by \cite{Szego:1975:OP} (and used in \cite{Dunster:1999:AEJ} and \cite{Elliott:1971:AEJ}), given by
\begin{multline}
\label{eq1.11}
Q_{n}^{(p,q)}(z)=\frac{1}{2}\Gamma(n+p+1)\Gamma(n+q+1)
\left ( \frac{2}{z-1} \right )^{n+p+q+1} \\
\times \textbf{F}
\left(n+p+q+1,n+q+1;2n+p+q+2;2(1-z)^{-1}\right),
\end{multline}
where $\mathbf{F}$ is Olver's scaled hypergeometric function defined by
\begin{equation}
\label{eq1.12}
\textbf{F}(a,b;c;z)=
\frac{F(a,b;c;z)}{\Gamma(c)}
=\sum_{s=0}^{\infty}\frac{(a)_{s}(b)_{s}}
{\Gamma(c+s)}\frac{z^{s}}{s!},
\end{equation}
which unlike the unscaled $F$ is defined for all values of the parameters. This hypergeometric function, as it appears in (\ref{eq1.11}), has branch points at $z= \pm 1$ and, in conjunction with the factor $(z-1)^{-n-p-q-1}$, takes its principal value in the $z$-plane having a cut along $(-\infty,1]$. The significance of $Q_{n}^{(p,q)}(z)$ is that it is recessive at $z=\infty$ in the principal plane.

With (\ref{eq1.4}) in mind we then define the companion to $C^{(\lambda)}_{n}(z)$ by
\begin{multline}
\label{eq1.13}
D_{n}^{(\lambda)}(z)
=\frac{(2\nu+1)_{n}}
{(\nu+1)_{n}}
Q_{n}^{(\nu,\nu)}(z)
=
\frac{k_{n}(\nu)}{
 \left\{2(z-1)\right\}^{2\nu+n+1}}
 \\
\times F\left(2\nu+n+1,\nu+n+1;
2\nu+2n+2;2(1-z)^{-1}\right),
\end{multline}
where
\begin{equation}
\label{eq1.14}
k_{n}(\nu) =
\frac{\pi \Gamma \left( 2\nu+n+1 \right)}{\Gamma \left( \nu+\tfrac{1}{2} \right)\Gamma
 \left( \nu+n+\tfrac{3}{2} \right)}.
\end{equation}
This too is recessive at $z=\infty$ in the principal plane. Specifically, from the behavior of $F$ for small argument we see from (\ref{eq1.13}) that
\begin{equation}
\label{eq1.15}
D_{n}^{(\lambda)}(z) =
\frac{k_{n}(\nu)}
{(2z)^{2\nu+n+1}}
\left\{1+\mathcal{O}\left(\frac{1}{z}\right)
\right\} \quad (z \rightarrow \infty,
\, |\arg(z)| < \pi),
\end{equation}
which should be compared to the dominant behavior of $C_{n}^{(\lambda)}(z)$ given by (\ref{eq1.7}).

Next, from (\ref{eq1.13}) and \cite[Eq. (1.13)]{Dunster:1999:AEJ}, we find as $z \rightarrow 1$ for $\nu>0$ that it is a dominant solution at this point, with the behavior
\begin{equation}
\label{eq1.16}
D_{n}^{(\lambda)}(z) \sim
\frac{\sqrt{\pi} \,
\Gamma(\nu)}
{2\Gamma(\nu+\tfrac{1}{2})}
\frac{1}{\left\{2(z - 1)\right\}^{\nu}}.
\end{equation}

We now introduce two more solutions, $D_{n,\pm 1}^{(\lambda)}(z)$ say, that are recessive as $z \rightarrow \pm i \infty$ across cut $[-1,1]$. We define these by
\begin{equation}
\label{eq1.17}
D_{n,\pm 1}^{(\lambda)}(z) =e^{\pm \nu \pi i }
D_{n}^{(\lambda)}\left(1+(z-1)e^{\pm 2\pi i}\right),
\end{equation}
where $1+(z-1)e^{\pm 2\pi i}$ represents the branch of the multi-valued function $D_{n}^{(\lambda)}(z)$ as $z$ encircles the branch point $z=1$ once in the positive (respectively negative) direction.

We note that
\begin{equation}
\label{eq1.17a}
\overline{D_{n,\pm 1}^{(\lambda)}(\bar{z})}
=D_{n,\mp 1}^{(\lambda)}(z).
\end{equation}
These two functions, along with $C_{n}^{(\lambda)}(z)$ and $D_{n}^{(\lambda)}(z)$, are solutions of the linear second order differential equation (\ref{eq1.2}), and hence are linear combinations of one another. These relationships are important, and from (\ref{eq1.4}), (\ref{eq1.13}) and \cite[Eq. (1.28)]{Dunster:1999:AEJ} are found to be
\begin{equation}
\label{eq1.18}
\cos(\nu \pi) D_{n}^{(\lambda)}(z) =
\tfrac{1}{2}
D_{n,1}^{(\lambda)}(z)
+ \tfrac{1}{2}D_{n,-1}^{(\lambda)}(z),
\end{equation}
\begin{equation}
\label{eq1.19}
\cos(\nu \pi) C_{n}^{(\lambda)}(z) =
\frac{i }{2\pi}  
\left\{
e^{\nu \pi i }
D_{n,1}^{(\lambda)}(z)
-e^{-\nu \pi i }
D_{n,-1}^{(\lambda)}(z)
\right\},
\end{equation}
and
\begin{equation}
\label{eq1.20}
D_{n,\pm 1}^{(\lambda)}(z)=e^{\mp \nu\pi i}
D_{n}^{(\lambda)}(z) \mp \pi i 
C_{n}^{(\lambda)}(z).
\end{equation}
Note as $z \rightarrow \infty$ in the principal plane we have from (\ref{eq1.7}), (\ref{eq1.15}) and (\ref{eq1.20})
\begin{equation}
\label{eq1.21}
D_{n,\pm 1}^{(\lambda)}(z)
=\mp 
\frac{\pi i(\lambda)_{n}(2z)^{n}}{n!}
\left\{1+\mathcal{O}\left(\frac{1}{z}\right)\right\}.
\end{equation}

The plan of this paper is as follows. In \cref{sec2} we obtain LG asymptotic expansions in terms of elementary functions that are valid in certain unbounded regions of the complex plane, but must exclude a neighborhood of the simple pole $z=1$. These expansions are obtained for $D_{n,}^{(\lambda)}(z)$ and $D_{n,\pm 1}^{(\lambda)}(z)$, as well as one other solution that we define and require later, using the expansions and error bounds given in \cite{Dunster:2020:LGE}. Unlike standard LG expansions, these asymptotic expansions have the coefficients appearing in the exponent, which in general makes their evaluation simpler, as well as providing simpler and readily computable error bounds.

The form of our LG expansions is also more amenable to turning point problems, as studied in \cite{Dunster:2017:COA}, and as we show in \cref{sec3} this is also true for equations having a simple pole. Thus, in \cref{sec3} we obtain exact expressions for all solutions of (\ref{eq1.1}), including of course the Gegenbauer polynomials, that involve modified Bessel functions and two slowly varying coefficient functions. The coefficients in the asymptotic expansions for these coefficient functions involve the ones appearing in our LG expansions, and so are not analytic at $z=1$. However, this is overcome in \cref{sec4}, where we specialize the results to the real interval $0 \leq z \leq 1$ by re-expanding in inverse powers of the large parameter, providing new coefficients which we prove to be analytic at the pole.

As in \cite{Dunster:2017:COA} we can compute the coefficients in \cref{sec3} for complex $z$ near or at $z=1$ without re-expanding this way, by appealing to Cauchy's integral formula. We do this in \cref{sec5}, in which simple error bounds are also constructed that only involve elementary functions. Our method of constructing asymptotic expansions and error bounds for asymptotic solutions of differential equations with a simple pole is quite general, and it is planned to apply it to other functions, such as conical functions.

The assumption that $n$ be an integer can be relaxed for our expansions of the companion functions, and indeed of the polynomial $C_{n}^{(\lambda)}(z)$ itself if one instead regards it as a hypergeometric function, via (\ref{eq1.3}). Our results can then be compared with a related result \cite[Eq. 15.12.4]{NIST:DLMF}.

\section{Expansions involving elementary functions} 
\label{sec2}

We now provide LG expansions involving elementary functions. These are useful in their own right but are particularly important in subsequent sections in our construction of expansions that are valid at the simple pole. We recast the differential equation in a normal form that is applicable to both the theory of this section as well as the next. To do so we redefine the dependent variable via
\begin{equation}
\label{eq2.1}
w(z)=\left(z^2-1\right)^{(\nu+1)/2}y(z),
\end{equation}
along with the large parameter
\begin{equation}
\label{eq2.2}
u=\lambda+n= \nu+n+\tfrac{1}{2}.
\end{equation}
Then (\ref{eq1.2}) becomes
\begin{equation}
\label{eq2.3}
w''(z)=\left\{u^2f(z)+g(z)\right\}w(z),
\end{equation}
where
\begin{equation}
\label{eq2.4}
f(z)=\frac{1}{z^2-1},
\end{equation}
and
\begin{equation}
\label{eq2.5}
g(z)=-\frac{1}{4 \left ( z^2-1 \right )}
+\frac{\nu^2 -1}
{\left ( z^2-1 \right )^2}.
\end{equation}

Note that the dominant term $f(z)$ has simple poles at $z= \pm 1$, and typically LG expansions break down at such points. Indeed the expansions we derive in this section are not valid at these points, but despite this we will use them for constructing expansions that are valid at $z=1$.

The choice (\ref{eq2.2}) for the form of the large parameter $u$, leading to the partition (\ref{eq2.4}) and (\ref{eq2.5}), is used in several papers for the more general Jacobi polynomials (see \cite{Dunster:1999:AEJ}, \cite{Frenzen:1985:AJP}), and is explained by making the subsequent asymptotic approximations valid at $z= \infty$; see also \cite[Remark 1.2]{Dunster:2020:LGE}.

Four numerically satisfactory solutions of (\ref{eq1.2}) are given by $(z^2-1)^{(\nu+1)/2}y(z)$, where $y(z)$ is $C^{(\lambda)}_{n}(z)$, $D^{(\lambda)}_{n}(z)$ or $D^{(\lambda)}_{n,\pm 1}(z)$. Before construction of our LG expansions, we identify another important solution, which will be useful in \cref{sec3}. This comes from observing that this differential equation is unchanged if $\nu$ and $n$ are replaced by $-\nu$ and $-n-1$, respectively. Thus from (\ref{eq1.3}), (\ref{eq1.5}), (\ref{eq1.12}) and (\ref{eq2.1}) we have a solution $(z^2-1)^{(\nu+1)/2}\hat{C}^{(\lambda)}_{n}(z)$ where
\begin{equation}
\label{eq3.8}
\hat{C}^{(\lambda)}_{n}(z)=
\frac{\pi^{1/2}}{\Gamma(\lambda)}
\left(z^2-1\right)^{-\nu}
\mathbf{F}\left(-2\nu-n,n+1;1-\nu;
\tfrac{1}{2}(1-z)\right).
\end{equation}
Here we have introduced the factor $\pi^{1/2}/\Gamma(\lambda)$ for later convenience.

The characterizing behavior of this function is that $(z-1)^{\nu}\hat{C}^{(\lambda)}_{n}(z)$ is analytic at $z=1$ for all $\nu$. We observe from (\ref{eq1.12}) and (\ref{eq3.8}) that as $z \rightarrow 1$
\begin{equation}
\label{eq3.9}
\hat{C}^{(\lambda)}_{n}(z)=
\frac{\pi^{1/2}}
{\Gamma(\nu+\frac{1}{2})\Gamma(1-\nu) \left\{ 2(z-1)\right\}^{\nu}}
\left\{1+\mathcal{O}(z-1) \right\}
\quad  (\nu \neq 1,2,3,\ldots).
\end{equation}
From (\ref{eq3.8}), \cite[Chap. 5, Eq. (9.05)]{Olver:1997:ASF} and the gamma function reflection formula \cite[Eq. 5.5.3	]{NIST:DLMF}
\begin{equation}
\label{eqGAMMA}
\Gamma(x)\Gamma(1-x)=\frac{\pi}{\sin(\pi x)},
\end{equation}
we have in the exceptional case the recessive behavior as $z \rightarrow 1$
\begin{equation}
\label{eq3.9a}
\hat{C}^{(\lambda)}_{n}(z)=
\frac{(2\nu+1)_{n}}{n!}
\left\{1+\mathcal{O}(z-1) \right\}
\quad  (\nu = 1,2,3,\ldots).
\end{equation}

As $z \rightarrow \infty$, $|\arg(z-1)|<\pi$ and $\nu>-n-\frac{1}{2}$ from (\ref{eq2.2}), (\ref{eq3.8}) and  \cite[Eq. 15.8.2]{NIST:DLMF}
\begin{equation}
\label{eq3.10}
\hat{C}^{(\lambda)}_{n}(z)=
\frac{(\lambda)_{n} (2z)^{n}}{n!}
\left\{1+\mathcal{O}\left(\frac{1}{z}\right)
\right\}.
\end{equation}

Next $\hat{C}^{(\lambda)}_{n}(z)$ can be expressed as a linear combination of $C^{(\lambda)}_{n}(z)$ and $D^{(\lambda)}_{n}(z)$. The constants in this relation can be found by letting $z \rightarrow 1$ and $z \rightarrow \infty$. Using (\ref{eq1.6}), (\ref{eq1.7}), (\ref{eq1.15}), (\ref{eq1.16}), (\ref{eq3.9}), (\ref{eqGAMMA}) and (\ref{eq3.10}) we determine these constants, and consequently the desired linear relationship can be established as
\begin{equation}
\label{eq3.11}
\hat{C}^{(\lambda)}_{n}(z)
=C^{(\lambda)}_{n}(z)+\frac{2 \sin(\nu \pi)}{\pi}
D^{(\lambda)}_{n}(z).
\end{equation}
From this, (\ref{eq1.18}) and (\ref{eq1.19}) we find that
\begin{equation}
\label{eq3.12}
\cos(\nu \pi)\hat{C}^{(\lambda)}_{n}(z)
=\frac{i}{2 \pi}
\left\{e^{-\nu \pi i } D^{(\lambda)}_{n,1}(z)
-e^{\nu \pi i } D^{(\lambda)}_{n,-1}(z)
\right\}.
\end{equation}

Let us now construct the LG expansions for solutions of (\ref{eq2.3}). From (\ref{eq2.4}), (\ref{eq2.5}) and \cite[Eqs. (1.3) and (1.6)]{Dunster:2020:LGE} we find that the appropriate transformed independent variable is given by
\begin{equation}
\label{eq2.6}
\xi =\int_{1}^{z} f^{1/2}(t) dt
=\ln\left(z + \sqrt{z^2 - 1}\right),
\end{equation}
with branches taken so that it is positive when $z>1$ and by continuity elsewhere in the plane having a cut along the interval $-\infty < z \leq 1$.

The corresponding Schwarzian derivative is
\begin{equation}
\label{eq2.7}
\psi(\xi) =\frac{4f(z) {f}^{\prime \prime }(z) -5 
f^{\prime 2}(z) }{16f^{3}(z) }+\frac{g(z) }{f(z) }
=\frac{1-4\nu^2}{4\left(1-z^2\right)}.
\end{equation}

Now apply \cite[Eqs. (1.10) - (1.12)]{Dunster:2020:LGE} to obtain the coefficients in our expansions
\begin{equation}
\label{eq2.9}
E_{s}(\xi) =\int F_{s}(\xi) d \xi
\quad (s=1,2,3,\ldots),
\end{equation}
where 
\begin{equation}
\label{eq2.10}
F_{1}(\xi) ={\tfrac{1}{2}}\psi(\xi), 
\quad
F_{2}(\xi) =-{\tfrac{1}{4}}{\psi }^{\prime }(\xi),
\end{equation}
and
\begin{equation}
\label{eq2.11}
F_{s+1}(\xi) =-{\frac{1}{2}}
\frac{dF_{s}(\xi) }{d \xi}-{\frac{1}{2}}
\sum\limits_{j=1}^{s-1}{F_{j}(\xi) F_{s-j}(\xi) }
\quad (s=2,3,4,\ldots).
\end{equation}

As we shall shortly see, all integrations can be evaluated more easily in terms of the variable defined by
\begin{equation}
\label{eq2.12}
\beta=\frac{z}{\sqrt{z^2 - 1}},
\end{equation}
or equivalently, from (\ref{eq2.12}),
\begin{equation}
\label{eq2.13}
\xi=\frac{1}{2}\ln\left(\frac{\beta+1}{\beta-1}\right).
\end{equation}
In both principal branches of the multi-valued functions apply. Now from (\ref{eq2.13}) we have
\begin{equation}
\label{eq2.12a}
\frac{d \xi}{d \beta}
=\frac{1}{1-\beta^2},
\end{equation}
and hence from (\ref{eq2.7}) and (\ref{eq2.10}) - (\ref{eq2.12}) we find that each $F_s(\xi)$ regarded as a function of $\beta$, which we denote by $\tilde{F}_{s}(\beta)$, is a polynomial in $\beta$, with
\begin{equation}
\label{eq2.14}
\tilde{F}_{1}(\beta )
=\tfrac{1}{8}\left (4\nu^2-1 \right)\left (\beta^2-1 \right),
\end{equation}
\begin{equation}
\label{eq2.15}
\tilde{F}_{2}(\beta)=\tfrac{1}{8}
\left( 4\nu^2-1\right)\beta\left(\beta^2-1\right),
\end{equation}
and
\begin{equation}
\label{eq2.16}
\tilde{F}_{s+1}(\beta)
=\frac{1}{2}\left(\beta^2-1\right)
\frac{d\tilde{F}_{s}(\beta)}{d\beta}-{\frac{1}{2}}
\sum\limits_{j=1}^{s-1}
\tilde{F}_{j}(\beta)\tilde{F}_{s-j}(\beta)
\quad (s=2,3,4,\ldots).
\end{equation}
Consequently, from (\ref{eq2.9}) and (\ref{eq2.12a}), and the fact that $\tilde{F}_{s}(\pm 1)=0$ ($s=1,2,3,\ldots$) which follows from induction on (\ref{eq2.14}) - (\ref{eq2.16}), we find that the coefficients $E_{s}(\xi)$, regarded as functions of $\beta$ and denoted by $\tilde{E}_{s}(\beta)$, are also polynomials in $\beta$. Specifically, they are given by
\begin{equation}
\label{eq2.17}
\tilde{E}_{s}(\beta) =-\int_{0}^{\beta}
\frac{\tilde{F}_{s}(b)}{b^2-1} d b
\quad (s=1,2,3,\ldots).
\end{equation}

Here we have chosen the arbitrary integration constants so that $\tilde{E}_{s}(0)=0$. Observe that $\tilde{E}_{2s}(\beta)$ are even and $\tilde{E}_{2s+1}(\beta)$ are odd functions of $\beta$. As we shall see, it is necessary that $\tilde{E}_{2s+1}(\beta)$ are odd functions of $\beta$, and the consequence of this is the following, which is readily verified from (\ref{eq2.12}) and (\ref{eq2.14}) - (\ref{eq2.17}).
\begin{lemma}
\label{lemMero}
Regarded as functions of $z$ via (\ref{eq2.12}), $(z-1)^{1/2}\tilde{E}_{2s-1}(\beta)$ and $\tilde{E}_{2s}(\beta)$ ($s=1,2,3,\ldots$) are meromorphic functions at $z=1$ 
\end{lemma}

We now are the position to apply \cite[Eq. (1.9) and Thm. 1.1]{Dunster:2020:LGE} to obtain LG asymptotic solutions of (\ref{eq2.3}) of the form
\begin{equation}
\label{eq2.18}
w^{\pm}(u,z) \sim f^{-1/4}(z)\exp \left\{\pm u\xi
+\sum\limits_{s=1}^{\infty}\frac{\tilde{E}_{s}(\beta) }
{(\pm u)^{s}} \right\}
\quad (u \rightarrow \infty).
\end{equation}

Before discussing regions of validity, let us identify these forms with our solutions of (\ref{eq2.3}). Firstly, we can assert that $(z^2-1)^{(\nu+1)/2}D^{(\lambda)}_{n}(z)$ and $w^{-}(u,z)$ are equal, to within a multiplicative constant, since both are the unique solutions of the same equation that are recessive at $z= +\infty$ ($\xi =+\infty$). Now, from (\ref{eq2.6}) and (\ref{eq2.12}) we see that $\xi=\ln(2z)+\mathcal{O}(z^{-2})$ and $\beta \rightarrow 1$ as $z \rightarrow +\infty$, and hence from (\ref{eq1.15}) we find the proportionality constant. As result we arrive at our first asymptotic expansion for large $u$, namely
\begin{equation}
\label{eq2.19}
D_{n}^{(\lambda)}(z) \sim 
\frac{k_{n}(\nu)}
{\left\{4\left(z^2-1\right)\right\}
^{(2\nu+1)/4}}
\exp \left\{ -u\xi
+\sum\limits_{s=1}^{\infty}(-1)^{s}\frac{\tilde{E}_{s}(\beta)
-\tilde{E}_{s}(1)}{u^{s}}
\right\}.
\end{equation}

Next let $z \rightarrow +\infty$ after crossing the cut $[-1,1]$ from above, and then continuing along the positive real axis; i.e.
\begin{equation*}
\label{eq2.19a}
1+(z-1)e^{2\pi i} \rightarrow +\infty.
\end{equation*}
Then from (\ref{eq2.6}) under this limit we have
\begin{equation}
\label{eq2.20}
\xi = -\ln(2|z|) +\mathcal{O}\left(z^{-3}\right),
\end{equation}
and from (\ref{eq2.12}) we correspondingly find that $\beta \rightarrow -1$.

In this case $D_{n,-1}^{(\lambda)}(z)$ is recessive and we find from (\ref{eq1.15}), (\ref{eq1.17}), (\ref{eq2.18}) and (\ref{eq2.20})
\begin{equation}
\label{eq2.21}
D_{n,-1}^{(\lambda)}(z) \sim 
\frac{i k_{n}(\nu)}
{\left\{4\left(z^2-1\right)\right\}
^{(2\nu+1)/4}}
\exp \left\{u\xi
+\sum\limits_{s=1}^{\infty}\frac{\tilde{E}_{s}(\beta)
-\tilde{E}_{s}(-1)}{u^{s}}
\right\}.
\end{equation}

Similarly, for the solution that is recessive as $z \rightarrow +\infty$ after crossing the cut $[-1,1]$ from below, we obtain
\begin{equation}
\label{eq2.22}
D_{n,1}^{(\lambda)}(z) \sim 
-\frac{i k_{n}(\nu)}
{\left\{4\left(z^2-1\right)\right\}
^{(2\nu+1)/4}}
\exp \left\{u\xi
+\sum\limits_{s=1}^{\infty}\frac{\tilde{E}_{s}(\beta)
-\tilde{E}_{s}(-1)}{u^{s}}
\right\}.
\end{equation}

Although the expansions in (\ref{eq2.21}) and (\ref{eq2.22}) differ only by a sign, the two functions are linearly independent, and these expansions have different domains of validity, as we discuss shortly.

Error bounds for these expansions are given in \cref{sec5}, and this theory also gives precise regions of validity, which include the points at infinity that we used in their derivation. For our purposes we are primarily interested in the half plane $|\arg(z)| \leq \pi/2$ since (\ref{eq2.22b}) and the following connection formula 
\begin{equation}
\label{eq2.22c}
D^{(\lambda)}_{n}\left(z e^{\pm \pi i}\right)=
(-1)^{n+1} e^{\mp 2 \nu \pi i} D^{(\lambda)}_{n}(z),
\end{equation}
which follows from (\ref{eq1.13}) and \cite[Eq. (1.24)]{Dunster:1999:AEJ}, extends our results to other values of $z$. 

As we show in \cref{sec5} the expansion (\ref{eq2.19}) is valid in a domain that contains all points in the half plane $|\arg(z)| \leq \pi/2$ except the singularity $z=1$, including all points above and below the cut $0 \leq z < 1$. Likewise, the bounds show that the expansion (\ref{eq2.21}) is valid in a domain that contains the first quadrant $0 \leq \arg(z) \leq \pi/2$, again with $z=1$ excluded, as well as points accessed by crossing the cut $[0,1]$ from above onto the entire negative imaginary axis $z=-i y$, $0 \leq y < \infty$. The domain of validity for (\ref{eq2.22}) is the complex conjugate of this, which includes the fourth quadrant $-\pi/2 \leq \arg(z) \leq 0$ with $z=1$ removed, and all points accessed by crossing the cut $[0,1]$ from below onto the entire positive imaginary axis $z=i y$, $0 \leq y < \infty$.

Although points arbitrarily close to but not equal to $z=1$ are theoretically included in the above domains of asymptotic validity, the expansions of course become less accurate for $z$ too close to this singularity, since in this case $\beta \rightarrow \infty$. For this reason we restrict use of the above LG expansions to $|z-1| \geq 1$, which we describe in more detail in \cref{sec5}.

For the first quadrant the corresponding (compound) LG expansion for $C_{n}^{(\lambda)}(z)$ follows from inserting (\ref{eq2.19}) and (\ref{eq2.21}) into (\ref{eq1.20}) (with the lower signs), and likewise for the fourth quadrant one can use (\ref{eq1.20}) (with the upper signs), (\ref{eq2.19}) and (\ref{eq2.22}). Alternatively, in the fourth quadrant one could use the results for the first quadrant along with Schwarz symmetry $\overline{C_{n}^{(\lambda)}(\bar{z})}=C_{n}^{(\lambda)}(z)$.

We finish this section by obtaining a useful formal expansion involving the constants $\tilde{E}_{s}(1)$ and the gamma function. To do so, we use
\begin{equation}
\label{eq2.22a}
\pi i C_{n}^{(\lambda)}(0) =
D_{n,-1}^{(\lambda)}(0+i0)
-e^{\nu\pi i} D_{n}^{(\lambda)}(0+i0),
\end{equation}
which comes from from (\ref{eq1.20}). Here $\pm i0$ denotes a point above (respectively below) the cut $(-\infty,1]$; note that $C_{n}^{(\lambda)}(0 \pm i0)=C_{n}^{(\lambda)}(0)$ since it is entire.

If we assume temporarily that $n$ is even, then from (\ref{eq2.1})
\begin{equation}
\label{eq2.24}
C^{(\lambda)}_{n}(0)=
(-1)^{n/2}\frac{\Gamma \left(\lambda+\tfrac{1}{2}n \right)}
{\Gamma (\lambda )\Gamma\left(\tfrac{1}{2}n+1\right)}.
\end{equation}
From (\ref{eq2.6})
\begin{equation*}
\label{eq2.23}
z=0+i0 \implies \xi=i \pi /2,
\end{equation*}
hence from (\ref{eq2.19}) and (\ref{eq2.21})
\begin{equation}
\label{eq2.25}
D_{n}^{(\lambda)}(0+i0) \sim 
(-1)^{(n/2)+1}
i e^{-\nu \pi i} 2^{-\nu-(1/2)} k_{n}(\nu)
\exp \left\{ 
\sum\limits_{s=1}^{\infty}(-1)^{s+1}
\frac{\tilde{E}_{s}(1)}{u^{s}}
\right\},
\end{equation}
and
\begin{equation}
\label{eq2.26}
D_{n,-1}^{(\lambda)}(0+i0) \sim 
(-1)^{n/2}i 2^{-\nu-(1/2)} k_{n}(\nu)
\exp \left\{
\sum\limits_{s=1}^{\infty}(-1)^{s+1}
\frac{\tilde{E}_{s}(1)}{u^{s}}
\right\}.
\end{equation}
Our desired relation then comes from (\ref{eq1.14}), (\ref{eq2.2}) and (\ref{eq2.22a}) - (\ref{eq2.26}), and reads as follows:
\begin{equation}
\label{eq2.27}
\frac {2^{\lambda-1}\Gamma (u+1) 
\Gamma \left(\frac {1}{2}u+
\frac{1}{2}\lambda\right)}
{\Gamma(u+\lambda ) 
\Gamma \left(\frac {1}{2}u-\frac {1}{2}\lambda+1
 \right) }
\sim \exp \left\{
\sum\limits_{s=1}^{\infty}(-1)^{s+1}
\frac{\tilde{E}_{s}(1)}{u^{s}}
\right\}  \quad (u \rightarrow \infty).
\end{equation}
By continuity we can now relax the assumption that $n$ be even.

\subsection{Modified Bessel functions}

We need similar expansions for modified Bessel functions, which we use in the next section for approximations to solutions of (\ref{eq2.3}) which are valid at the simple pole $z=1$. The functions in question are solutions of the equation
\begin{equation} 
\label{eq2.28}
\frac {d ^2 w}{d z^{2}} = \left\{ 1+
\frac {\nu^{2}-\frac{1}{4}}{z^{2}} \right\} w,
\end{equation}
given by $w=z^{1/2}K_{\nu}(z), \, z^{1/2}K_{\nu}(ze^{\pm \pi i})$. Here $K_{\nu}(z)$ is the modified Bessel function of the second kind, with the recessive property at infinity given by \cite[Eq. 10.25.3]{NIST:DLMF}
\begin{equation} 
\label{eq2.29}
K_{\nu}(z)\sim
\left(\frac{\pi}{2z}\right)^{1/2}
e^{-z}
\quad   \left(z \rightarrow \infty, \,
|\arg(z)| \leq \tfrac{3}{2} \pi - \delta\right).
\end{equation}

The other two functions $K_{\nu}(ze^{\pm \pi i})$ are analytic continuations, with recessive properties at infinity in the sectors indicated in (\ref{eq2.29}) rotated by $\mp \pi$. From \cite[Eq. 10.34.4]{NIST:DLMF} the three are related by
\begin{equation} 
\label{eq2.33}
\cos(\nu \pi)K_{\nu }(z)=\tfrac{1}{2}K_{\nu}\left ( ze^{\pi i } \right )+\tfrac{1}{2}K_{\nu}\left ( ze^{-\pi i } \right ).
\end{equation}

We now simply follow steps above in deriving (\ref{eq2.19}), but with $\xi=z$, $f(z)=1$ and $g(z)=(\nu^2-\frac{1}{4})z^{-2}$. We omit details since they are straightforward, and we obtain
\begin{multline} 
\label{eq2.30}
K_{\nu}(z)\sim
\left(\frac{\pi}{2z}\right)^{1/2}
\exp\left\{-z-\sum_{s=1}^{\infty} (-1)^s
\frac{a_{s}(\nu)}{s z^{s}}
\right\}
\\
\quad   \left(z \rightarrow \infty, \,
|\arg(z)| \leq \tfrac{3}{2} \pi - \delta\right),
\end{multline}
where
\begin{equation} 
\label{eq2.31}
a_{1}(\nu) =  a_{2}(\nu)
=\tfrac{1}{8}
\left(4\nu^{2}-1\right),
\end{equation}
and
\begin{equation} 
\label{eq2.32}
a_{s+1}(\nu)=\frac{1}{2}(s+1)a_{s}(\nu)
-\frac{1}{2}
\sum_{j=1}^{s-1}
a_{j}(\nu)a_{s-j}(\nu).
\end{equation}

From this we immediately deduce that
\begin{multline} 
\label{eq2.32a}
K_{\nu}\left(z e^{\pm \pi i}\right) \sim
\mp i \left(\frac{\pi}{2z}\right)^{1/2}
\exp\left\{z-\sum_{s=1}^{\infty} 
\frac{a_{s}(\nu)}{s z^{s}}
\right\} \\
\quad   \left(z \rightarrow \infty, \,
\left|\arg\left(z e^{\pm \pi i}\right)\right|
\leq \tfrac{3}{2} \pi - \delta\right).
\end{multline}

We note that the coefficients in (\ref{eq2.30}) differ from the explicitly-given ones that appear outside the exponential function in the standard form (see for example \cite[Eqs. 10.40.2 and 10.40.9]{NIST:DLMF}).

\section{Expansions in domains containing the simple pole} 
\label{sec3}

Let us now turn to obtaining the main results of this paper, namely expansions that are valid at the simple pole $z=1$. This is a modification of \cite[Chap. 12, Sect. 2.1]{Olver:1997:ASF}, and referring to (\ref{eq2.6}) we define the Liouville variable applicable here by
\begin{equation}
\label{eq3.1}
\zeta =\xi^2=
\left\{\ln\left(z + \sqrt{z^2 - 1}\right)\right\}^2.
\end{equation}
In contrast to $\xi$ this variable is analytic at $z=1$, and from (\ref{eq3.2}) we find that
\begin{equation}
\label{eq3.2}
\zeta = 2(z - 1) - \tfrac{1}{3}(z - 1)^2
+\mathcal{O}\left\{(z-1)^3\right\}
\quad (z \rightarrow 1).
\end{equation}
For the important oscillatory interval $-1\leq z\leq 1$ we find from (\ref{eq2.4}), (\ref{eq2.6}) and (\ref{eq3.1}) that $\zeta \leq 0$, with
\begin{equation}
\label{eq3.3}
(-\zeta)^{1/2} =\int_{z}^{1} \left\{-f(t) \right\}
^{1/2} dt
=\arccos{(z)}.
\end{equation}

Following the transformation \cite[Chap 12, Eq (2.06)]{Olver:1997:ASF} applied to our differential equation (\ref{eq2.3}) - (\ref{eq2.5}) we obtain the appropriate new equation
\begin{equation}
\label{eq3.4}
\frac{d^{2}W}{d\zeta^{2}}=\left\{\frac{u^{2}}{4\zeta }+ \frac{\nu ^{2}-1}{4\zeta^{2}}
+\frac{\tilde{\psi}(\zeta )}{\zeta }\right\}W,
\end{equation}
where
\begin{equation}
\label{eq3.5}
W(\zeta)=\left\{\zeta f(z)\right\}^{1/4}w(z)
= \left(\frac{\zeta}{z^2-1}\right)^{1/4}w(z),
\end{equation}
and
\begin{equation}
\label{eq3.6}
\tilde{\psi}(\zeta )
=\frac{1-4\nu^2}{16\zeta}
+\frac{g(z) }{4f(z) }
+\frac{4f(z) {f}^{\prime \prime }(z) 
-5 f^{\prime 2}(z) }{64f^{3}(z) }
=\frac{1-4\nu^2}{16}\left( \frac{1}{1-z^2}
+\frac{1}{\zeta} \right).
\end{equation}

In contrast to $\psi(\xi)$, given by (\ref{eq2.7}), $\tilde{\psi}(\zeta )$ is analytic at $z=1$ ($\zeta=0$), and in particular one finds that
\begin{equation*}
\label{eq3.6a}
\tilde{\psi}(\zeta )
=\frac{1-4\nu^2}{16}\left[
\frac{1}{3}-\frac{2}{15}(z-1) 
+\mathcal{O}\left\{(z-1)^2\right\}\right]
\quad (z \rightarrow 1).
\end{equation*}

Note that the dominant term of (\ref{eq3.4}), namely the term involving the large parameter $u^2$, has a simple pole at $\zeta =0$, similarly to its parent (\ref{eq2.3}). Also observe that if we neglect the term $\tilde{\psi}(\zeta )/\zeta$ the equation has solutions $\zeta^{1/2}\mathcal{L}_{\nu}(u \zeta^{1/2})=\xi\mathcal{L}_{\nu}(u \xi)$, where $\mathcal{L}_{\nu}(z)$ is any modified Bessel function, including the three $K$ functions described earlier, as well as the modified Bessel function $I_{\nu}(z)$ which is recessive at $z=0$ for $\nu \geq 0$ (see \cite[Eq. 10.25.2]{NIST:DLMF}). Indeed these are our approximants, and the reader is referred to \cite[Chap. 12, Thm. 9.1]{Olver:1997:ASF} for more details.

In this reference Olver obtains asymptotic solutions of (\ref{eq3.4}) of the form
\begin{multline}
\label{eq3.6b}
W_{2n+1,j}(u,\zeta) =
\xi \mathcal{L}_{\nu}(u \xi)
\sum\limits_{s=0}^{n}
\frac{A_{s}(\zeta)}{u^{2s}} \\
+(-1)^{j+1} \frac{\xi^2}{u}\mathcal{L}_{\nu+1}(u \xi)
\sum\limits_{s=0}^{n-1}\frac{B_{s}(\zeta)}{u^{2s}}
+\epsilon_{2n+1,j}(u,\zeta)
\quad (j=1,2),
\end{multline}
where $\mathcal{L}=I$ and $\mathcal{L}=K$ for $j=1,2$ respectively. The coefficients $A_{s}(\zeta)$ and $B_{s}(\zeta)$ are analytic in the domain under consideration (which contains the pole $\zeta=0$), and satisfy certain recursion relations which involve repeated integration (rendering them typically hard to compute). Complicated bounds are provided for the error terms $\epsilon_{2n+1,j}(u,\zeta)$, and their derivatives, which involve so-called auxiliary functions for the modified Bessel functions.

We prefer to now return to the original variable $z$. On referring to (\ref{eq3.5}) and (\ref{eq3.6b}), and following \cite{Boyd:1990:CSP} and \cite{Dunster:2017:COA}, we seek   solutions of (\ref{eq2.3}) in the form
\begin{multline}
\label{eq3.15}
w_{j}(u,z) =\xi^{1/2}f^{-1/4}(z) 
\\ \times
\left\{\mathcal{L}_{\nu}^{(j)}
(u \xi) A(u,z) 
+(-1)^{j+1}\mathcal{L}_{\nu+1}^{(j)}(u \xi)
B(u,z) \right\}
\quad (j=0,\pm 1),
\end{multline}
where
\begin{equation} 
\label{eq3.13}
\mathcal{L}_{\nu}^{(0)}(z)
=K_{\nu }(z),
\end{equation}
and
\begin{equation} 
\label{eq3.14}
\mathcal{L}_{\nu}^{(\pm 1)}(z)
=K_{\nu}\left ( ze^{\pm \pi i } \right ).
\end{equation}

Here we aim for $A(u,z)$ and $B(u,z)$ to be slowly-varying functions of $u$ and analytic in a domain that contains $z=1$. With this in mind, we \emph{define} $A(u,z)$ and $B(u,z)$ by the pair of equations
\begin{multline}
\label{eq3.19}
D_{n,\pm 1}^{(\lambda)}(z)
=\frac{\sqrt{u\pi} 
(\lambda)_{n}}{2^{\nu} n!}
\left(z^2-1\right)^{-(\nu+1)/2} w_{\pm 1}(u,z)
\\
=\frac{\sqrt{u\pi} 
(\lambda)_{n}}{2^{\nu} n!}
\xi^{1/2}
\left(z^2-1\right)^{-(2\nu+1)/4} 
\\ \times
\left\{K_{\nu}\left ( u \xi e^{\pm \pi i } \right ) A(u,z) 
+ K_{\nu+1}\left(u \xi e^{\pm \pi i } \right ) B(u,z) \right\},
\end{multline}
where the factor $(z^2-1)^{-(\nu+1)/2}$ comes from (\ref{eq2.1}), and the multiplicative constant has been introduced for later convenience, but must be the same for both functions, as we see next.

Note that these are exact expressions, and so from these two equations we can use connection formulas to obtain corresponding representations for all other solutions of (\ref{eq1.2}). Firstly, using (\ref{eq1.18}) and (\ref{eq2.33}) we get from (\ref{eq3.19})
\begin{multline}
\label{eq3.22}
D_{n}^{(\lambda)}(z)
=\frac{\sqrt{u\pi} 
(\lambda)_{n}}{2^{\nu} n!}
\xi^{1/2}
\left(z^2-1\right)^{-(2\nu+1)/4} 
 \\ \times 
\left\{K_{\nu}(u \xi) A(u,z) 
-K_{\nu+1}(u \xi)B(u,z) \right\}.
\end{multline}

Next, using \cite[Sect. 34]{NIST:DLMF}, 
\begin{equation}
\label{eq3.23}
\cos(\nu \pi) I_{\nu }(z) =
\frac{i }{2\pi}  
\left\{
e^{\nu \pi i } K_{\nu}\left ( ze^{\pi i } \right )
- e^{-\nu \pi i } K_{\nu}\left ( ze^{-\pi i } \right )
\right\},
\end{equation}
and (\ref{eq1.19}) yields
\begin{multline}
\label{eq3.24}
C_{n}^{(\lambda)}(z)
=\frac{\sqrt{u\pi} 
(\lambda)_{n}}{2^{\nu} n!}
\xi^{1/2}
\left(z^2-1\right)^{-(2\nu+1)/4} 
 \\ \times
\left\{I_{\nu}(u \xi) A(u,z) 
+I_{\nu+1}(u \xi)B(u,z) \right\}.
\end{multline}
Similarly, we use (\ref{eq3.12}), (\ref{eq3.19}), and (\ref{eq3.23}) with $\nu$ replaced by $-\nu$, and we arrive at
\begin{multline}
\label{eq3.26}
\hat{C}^{(\lambda)}_{n}(z)
=\frac{\sqrt{u \pi}\,
(\lambda)_{n}}{2^{\nu} n!}
\xi^{1/2}
\left(z^2-1\right)^{-(2\nu+1)/4}
\\ \times
\left\{I_{-\nu}(u \xi) A(u,z) 
+I_{-\nu-1}(u \xi)B(u,z) \right\}.
\end{multline}

We shall make use of the following Wronskian relations, found from \cite[Sects. 10.28 and 10.34]{NIST:DLMF}, to aid in finding representations for the coefficient functions. These are given by
\begin{equation} 
\label{eq3.16}
\mathcal{L}_{\nu}^{(1)}(z)\mathcal{L}_{\nu+1}^{(-1)}(z)
-\mathcal{L}_{\nu}^{(-1)}(z)\mathcal{L}_{\nu+1}^{(1)}(z)
=2\pi i\cos(\nu \pi)/z,
\end{equation}
\begin{equation} 
\label{eq3.18}
\mathcal{L}_{\nu}^{(0)}(z)
\mathcal{L}_{\nu+1}^{(\pm 1)}(z)
+\mathcal{L}_{\nu+1}^{(0)}(z)
\mathcal{L}_{\nu}^{(\pm 1)}(z)
=\mp i\pi/z,
\end{equation}
and
\begin{equation} 
\label{eq3.17}
I_{\nu}(z)I_{-\nu-1}(z)-I_{\nu+1}(z)I_{-\nu}(z)
=-2\sin(\nu\pi)/(\pi z).
\end{equation}

Explicit representations for both coefficient functions can now be established by choosing any pair of equations from (\ref{eq3.19}) - (\ref{eq3.26}) and solving for $A(u,z)$ and $B(u,z)$. The choice of equations depends in which part of the complex plane one is considering. In the region under consideration we choose the pair of equations that involve numerically satisfactory solutions to (\ref{eq1.2}).

To this end, consider the first quadrant. Here a numerically satisfactory pair of solutions is $D_{n}^{(\lambda)}(z)$ and $D_{n,-1}^{(\lambda)}(z)$, provided points close to $z=1$ are excluded at this time. Thus we use (\ref{eq3.19}), (\ref{eq3.22}) and (\ref{eq3.18}) to obtain
\begin{multline}
\label{eq3.29}
A(u,z)=
-\frac{i 2^{\nu} n!}{\pi^{3/2}(\lambda)_{n}}
(u\xi)^{1/2}
\left(z^2-1\right)^{(2\nu+1)/4}  \\ \times
\left\{ D_{n}^{(\lambda)}(z)
K_{\nu+1}\left ( u \xi e^{-\pi i } \right )
+D_{n,-1}^{(\lambda)}(z)K_{\nu+1}(u \xi)\right\},
\end{multline}
and
\begin{multline}
\label{eq3.29a}
B(u,z)=
\frac{i 2^{\nu} n!}{\pi^{3/2}(\lambda)_{n}}
(u\xi)^{1/2}
\left(z^2-1\right)^{(2\nu+1)/4} \\ \times
\left\{
D_{n}^{(\lambda)}(z) 
K_{\nu}\left(u \xi e^{-\pi i } \right)
-D_{n,-1}^{(\lambda)}(z)K_{\nu}(u \xi)
\right\}.
\end{multline}

Similarly, in fourth quadrant we obtain the numerically satisfactory representations
\begin{multline}
\label{eq3.27}
A(u,z)=
\frac{i 2^{\nu} n!}{\pi^{3/2}(\lambda)_{n}}
(u\xi)^{1/2}
\left(z^2-1\right)^{(2\nu+1)/4} \\ \times
\left\{ D_{n}^{(\lambda)}(z) 
K_{\nu+1}\left ( u \xi e^{\pi i } \right )
+D_{n,1}^{(\lambda)}(z)K_{\nu+1}(u \xi)\right\},
\end{multline}
and
\begin{multline}
\label{eq3.28}
B(u,z)=
-\frac{i 2^{\nu} n!}{\pi^{3/2}(\lambda)_{n}}
(u\xi)^{1/2}
\left(z^2-1\right)^{(2\nu+1)/4} \\ \times
\left\{
D_{n}^{(\lambda)}(z) 
K_{\nu}\left(u \xi e^{\pi i } \right)
-D_{n,1}^{(\lambda)}(z)K_{\nu}(u \xi)
\right\}.
\end{multline}

These are our main representations which we shall use shortly to construct asymptotic expansions. Before doing so, other choice of pairs of equations are possible, and to prove analyticity we find from (\ref{eq3.24}), (\ref{eq3.26}) and (\ref{eq3.17}) that
\begin{multline}
\label{eq3.30}
\sin(\nu \pi) A(u,z)=
\frac{2^{\nu-1} n!}{(\lambda)_{n}}
(\pi u \xi)^{1/2}
\left(z^2-1\right)^{(2\nu+1)/4}  \\ \times
\left\{ \hat{C}_{n}^{(\lambda)}(z)I_{\nu+1}(u\xi)
-C_{n}^{(\lambda)}(z)I_{-\nu-1}(u \xi)\right\},
\end{multline}
and
\begin{multline}
\label{eq3.30a}
\sin(\nu \pi) B(u,z)=
-\frac{2^{\nu-1} n!}{(\lambda)_{n}}
(\pi u \xi)^{1/2}
\left(z^2-1\right)^{(2\nu+1)/4}  \\ \times
\left\{ \hat{C}_{n}^{(\lambda)}(z)I_{\nu}(u\xi)
-C_{n}^{(\lambda)}(z)I_{-\nu}(u \xi)\right\}.
\end{multline}

Assuming temporarily that $\nu$ is not an integer, using (\ref{eq1.1}), (\ref{eq3.8}), (\ref{eq3.1}), (\ref{eq3.2}), (\ref{eq3.30}) and \cite[Eq. 10.25.2]{NIST:DLMF} one can construct a convergent power series for $A(u,z)$ at $z=1$, demonstrating that it is analytic at that point. The restriction that $\nu$ be an integer can be relaxed by a limit argument. Similarly $B(u,z)$ can also shown to be analytic at $z=1$ for all $\nu$ from (\ref{eq3.30a}).

We also note
\begin{multline}
\label{eq3.31}
A(u,z)=
\frac{2^{\nu} n!}{(\lambda)_{n}}
\left(\frac{u\xi}{\pi}\right)^{1/2}
\left(z^2-1\right)^{(2\nu+1)/4}  \\ \times
\left\{ C_{n}^{(\lambda)}(z)K_{\nu+1}(u\xi)
+D_{n}^{(\lambda)}(z)I_{\nu+1}(u \xi)\right\},
\end{multline}
and
\begin{multline}
\label{eq3.31b}
B(u,z)=
\frac{2^{\nu} n!}{(\lambda)_{n}}
\left(\frac{u\xi}{\pi}\right)^{1/2}
\left(z^2-1\right)^{(2\nu+1)/4}  \\ \times
\left\{ C_{n}^{(\lambda)}(z)K_{\nu}(u\xi)
-D_{n}^{(\lambda)}(z)I_{\nu}(u \xi)\right\},
\end{multline}
which, in conjunction with their analyticity at $z=1$, show both of these functions are real for $z \in (-1,\infty)$.

Let us now use the results of \cref{sec2} to derive the desired expansions for $A(u,z)$ and $B(u,z)$. From (\ref{eq2.19}), (\ref{eq2.21}), (\ref{eq2.22}), (\ref{eq2.30}), (\ref{eq2.32a}), (\ref{eq3.29}) and (\ref{eq3.29a}) we obtain asymptotic expansions which are uniformly valid for all points in the first quadrant, except for an arbitrarily small neighborhood of $z=1$. Thus as $u \rightarrow \infty$ we arrive at
\begin{multline}
\label{eq3.32}
A(u,z) \sim 
\frac{\Gamma(u+\lambda) n!}
{\Gamma(u)\Gamma(u+1)}
\exp \left\{
\sum\limits_{s=1}^{\infty}
\frac{\tilde{\mathcal{E}}_{2s}(\nu,z)}{u^{2s}}
+\sum\limits_{s=1}^{\infty}(-1)^{s+1}
\frac{\tilde{E}_{s}(1)}{u^{s}}
\right\}
\\ \times
\cosh \left\{ \sum\limits_{s=0}^{\infty }
\frac{\tilde{\mathcal{E}}_{2s+1}(\nu,z)}{u^{2s+1}}
\right\},
\end{multline}
and
\begin{multline}
\label{eq3.33}
B(u,z) \sim 
\frac{\Gamma(u+\lambda) n!}
{\Gamma(u)\Gamma(u+1)}
\exp \left\{
\sum\limits_{s=1}^{\infty}
\frac{\mathcal{E}_{2s}(\nu,z)}{u^{2s}}
+\sum\limits_{s=1}^{\infty}(-1)^{s+1}
\frac{\tilde{E}_{s}(1)}{u^{s}}
\right\}
\\ \times
\sinh \left\{ \sum\limits_{s=0}^{\infty }
\frac{\mathcal{E}_{2s+1}(\nu,z)}{u^{2s+1}}
\right\},
\end{multline}
where for $s=1,2,3,\ldots$
\begin{equation}
\label{eq3.34}
\mathcal{E}_{s}(\nu,z)
=\tilde{E}_{s}(\beta)+(-1)^{s+1}\frac{a_{s}(\nu)}
{s\xi^{s}},
\end{equation}
and
\begin{equation}
\label{eq3.34a}
\tilde{\mathcal{E}}_{s}(\nu,z)
=\tilde{E}_{s}(\beta)+(-1)^{s+1}\frac{a_{s}(\nu+1)}
{s\xi^{s}}.
\end{equation}

From \cref{lemMero} and (\ref{eq2.6}) (see also (\ref{eq3.1}) and (\ref{eq3.2})) we have the following important result.
\begin{lemma}
\label{lemMero2}
$(z-1)^{1/2}\mathcal{E}_{2s+1}(\nu,z)$, $(z-1)^{1/2}\tilde{\mathcal{E}}_{2s+1}(\nu,z)$, $\mathcal{E}_{2s+2}(\nu,z)$ and $\tilde{\mathcal{E}}_{2s+2}(\nu,z)$ ($s=0,1,2,\ldots$) are meromorphic functions at $z=1$. 
\end{lemma}

If we use (\ref{eq3.27}) and (\ref{eq3.28}) in place of (\ref{eq3.29}) and (\ref{eq3.29a}) we obtain exactly the same expansions, but now also valid in the fourth quadrant, again with a neighborhood of $z=1$ removed. Thus the expansions hold uniformly for all points in the right half plane, except for a neighborhood of $z=1$.

Let us summarize our main results. In doing so we use (\ref{eq2.27}) to replace the series involving $\tilde{E}_{s}(1)$ in (\ref{eq3.32}) and (\ref{eq3.33}) with gamma functions.
\begin{theorem}
Let $u=\lambda + n$ and $\nu = \lambda -\frac{1}{2}$, and $\beta$ and $\xi$ be given by (\ref{eq2.12}) and (\ref{eq2.13}), respectively. Then the Gegenbauer polynomial $C_{n}^{(\lambda)}(z)$ can be expressed by
\begin{multline}
\label{eq3.24a}
C_{n}^{(\lambda)}(z)
=\frac{\sqrt{u\pi} 
(\lambda)_{n}}{2^{\nu} n!}
\xi^{1/2}
\left(z^2-1\right)^{-(2\nu+1)/4} 
 \\ \times
\left\{I_{\nu}(u \xi) A(u,z) 
+I_{\nu+1}(u \xi)B(u,z) \right\},
\end{multline}
with corresponding representations for companion solutions to (\ref{eq1.2}) given by (\ref{eq3.19}), (\ref{eq3.22}) and (\ref{eq3.26}). In these the fractional powers take positive values when $z>1$ ($\xi >0$) and are defined by continuity elsewhere. The functions $A(u,z)$ and $B(u,z)$ are analytic for $\Re(z) \geq 0$, and as $u \rightarrow \infty$, uniformly for $\Re(z) \geq 0$ and $|z-1| \geq \delta$ ($0< \delta \leq 1$), possess the asymptotic expansions
\begin{multline}
\label{eq3.35}
A(u,z) \sim 
\frac{2^{\lambda - 1}
\Gamma\left(\frac{1}{2}u+\frac{1}{2}\lambda\right)n!}
{\Gamma(u)
\Gamma\left(\frac{1}{2}u - \frac{1}{2}\lambda + 1\right)}
\exp \left\{
\sum\limits_{s=1}^{\infty}
\frac{\tilde{\mathcal{E}}_{2s}(\nu,z)}{u^{2s}}
\right\}
\\ \times
\cosh \left\{ \sum\limits_{s=0}^{\infty }
\frac{\tilde{\mathcal{E}}_{2s+1}(\nu,z)}{u^{2s+1}}
\right\},
\end{multline}
and
\begin{multline}
\label{eq3.36}
B(u,z) \sim 
\frac{2^{\lambda - 1}
\Gamma\left(\frac{1}{2}u+\frac{1}{2}\lambda\right)n!}
{\Gamma(u)
\Gamma\left(\frac{1}{2}u - \frac{1}{2}\lambda + 1\right)}
\exp \left\{
\sum\limits_{s=1}^{\infty}
\frac{\mathcal{E}_{2s}(\nu,z)}{u^{2s}}
\right\}
\\ \times
\sinh \left\{ \sum\limits_{s=0}^{\infty }
\frac{\mathcal{E}_{2s+1}(\nu,z)}{u^{2s+1}}
\right\}.
\end{multline}
The coefficients $\mathcal{E}_{s}(\nu,z)$ and $\tilde{\mathcal{E}}_{s}(\nu,z)$ are given by (\ref{eq3.34}) and (\ref{eq3.34a}), where $\tilde{E}_{s}(\beta)$ are defined by (\ref{eq2.14}) - (\ref{eq2.17}), and $a_{s}(\nu)$ are defined by (\ref{eq2.31}) and (\ref{eq2.32}). 
\end{theorem}

The first two coefficients in (\ref{eq3.35}) and (\ref{eq3.36}) are given by
\begin{equation}
\label{eq3.36a}
\mathcal{E}_{1}(\nu,z)
=\frac{4\nu^{2}-1}{8 \xi}
-\frac{\left (4\nu^2-1\right)\beta}{8},
\end{equation}
\begin{equation}
\label{eq3.36c}
\mathcal{E}_{2}(\nu,z)
=-\frac{4\nu^{2}-1}{16 \xi^{2}}
-\frac{\left (4\nu^2-1\right)\beta^2}{16},
\end{equation}
\begin{equation}
\label{eq3.36b}
\tilde{\mathcal{E}}_{1}(\nu,z)
=\frac{4(\nu+1)^{2}-1}{8 \xi}
-\frac{\left (4\nu^2-1\right)\beta}{8},
\end{equation}
and
\begin{equation}
\label{eq3.36d}
\tilde{\mathcal{E}}_{2}(\nu,z)
=-\frac{4(\nu+1)^{2}-1}{16 \xi^{2}}
-\frac{\left (4\nu^2-1\right)\beta^2}{16}.
\end{equation}

Extension of these results to $|z-1| \leq \delta$ will be discussed in the next two sections.

\section{Expansions for Gegenbauer polynomials with real argument $z \in [0,1]$}
\label{sec4}

Here we specialize results for the previous section to the real interval $[0,1]$ where the Gegenbauer polynomials have their zeros in the right half plane; the interval $[-1,0]$ containing the other zeros is simply covered by (\ref{eq2.22b}). Now this interval lies on the branch cut for the multi-valued functions and variables we have studied, and we only need to consider the upper part since obviously the polynomials themselves take the same value above and below.

Thus, let $z=x+i 0$ where $0 \leq x=\cos(\theta) <1$. Then we have from (\ref{eq2.6}) and (\ref{eq2.12})
\begin{equation}
\label{eq4.1}
\beta=-i\cot(\theta),
\end{equation}
and
\begin{equation}
\label{eq4.2}
\xi=i \theta.
\end{equation}

Now from \cite[Eq. 10.27.6]{NIST:DLMF}
\begin{equation*}
\label{eq4.3}
I_{\nu}(iz)=e^{\nu \pi i/2}J_{\nu }(z),
\end{equation*}
and so from this, (\ref{eq3.34}) - (\ref{eq3.36}), (\ref{eq4.1}) and (\ref{eq4.2}), we obtain
\begin{multline}
\label{eq4.4}
C_{n}^{(\lambda)}(\cos(\theta))
=\sqrt{\frac{u\pi}{2}}\,
\frac{ \Gamma\left(\frac{1}{2}u+\frac{1}{2}\lambda\right)
\theta^{1/2}}
{\Gamma\left(\frac{1}{2}u-\frac{1}{2}\lambda+1\right)
\Gamma(\lambda)
\{\sin(\theta) \}^{\lambda}}
\\ \times
\left\{J_{\nu}(u \theta) \hat{A}(u,\theta) 
-J_{\nu+1}(u \theta)\hat{B}(u,\theta) \right\},
\end{multline}
where
\begin{multline}
\label{eq4.5}
\hat{A}(u,\theta) =
\frac{\Gamma(u)
\Gamma\left(\frac{1}{2}u - \frac{1}{2}\lambda + 1\right)}
{2^{\lambda - 1}
\Gamma\left(\frac{1}{2}u+\frac{1}{2}\lambda\right)n!}
A(u,\cos(\theta))  
\\ 
\sim \exp \left\{
\sum\limits_{s=1}^{\infty}
\frac{\tilde{\mathbf{E}}_{2s}(\nu,\theta)}{u^{2s}}
\right\}
\cos \left\{ \sum\limits_{s=0}^{\infty }
\frac{\tilde{\mathbf{E}}_{2s+1}(\nu,\theta)}{u^{2s+1}}
\right\},
\end{multline}
and
\begin{multline}
\label{eq4.6}
\hat{B}(u,\theta)
=-i \frac{\Gamma(u)
\Gamma\left(\frac{1}{2}u - \frac{1}{2}\lambda + 1\right)}
{2^{\lambda - 1}
\Gamma\left(\frac{1}{2}u+\frac{1}{2}\lambda\right)n!}
B(u,\cos(\theta))
\\  \sim
\exp \left\{
\sum\limits_{s=1}^{\infty}
\frac{\mathbf{E}_{2s}(\nu,\theta)}{u^{2s}}
\right\}
\sin \left\{ \sum\limits_{s=0}^{\infty }
\frac{\mathbf{E}_{2s+1}(\nu,\theta)}{u^{2s+1}}
\right\},
\end{multline}
as $u \rightarrow \infty$ with $0<\delta \leq \theta \leq \pi/2$. Here
\begin{equation}
\label{eq4.7}
\mathbf{E}_{2s}(\nu,\theta)
=\tilde{E}_{2s}(i\cot(\theta))
+(-1)^{s+1}\frac{a_{2s}(\nu)}
{2 s\theta^{2s}},
\end{equation}
and
\begin{equation}
\label{eq4.8}
\mathbf{E}_{2s+1}(\nu,\theta)
=i\tilde{E}_{2s+1}(i\cot(\theta))
+(-1)^{s+1}\frac{a_{2s+1}(\nu)}
{(2s+1)\theta^{2s+1}},
\end{equation}
with $\tilde{\mathbf{E}}_{s}(\nu,\theta)$ defined the same way except with $a_{s}(\nu)$ replaced by $a_{s}(\nu+1)$.

We now re-expand $\hat{A}(u,\theta)$ and $\hat{B}(u,\theta)$ as standard asymptotic expansions in inverse powers of $u$. To do so, first consider the formal expansion
\begin{equation}
\label{eq4.9}
\exp \left\{
\sum\limits_{s=1}^{\infty}
\frac{\iota_{s}\tilde{\mathbf{E}}_{s}(\nu,\theta)}
{u^{s}}
\right\}
\sim 1+ \sum\limits_{s=1}^{\infty}
\frac{\iota_{s}\mathrm{A}_{s}(\theta)}{u^{s}},
\end{equation}
where $\iota_{s}=1$ for $s$ even and $\iota_{s}=i$ for $s$ odd. The real part of both sides is equivalent to the asymptotic expansion in (\ref{eq4.5}).

To determine the coefficients $\mathrm{A}_{s}(\theta)$ we differentiate (\ref{eq4.9}) with respect to $u$, and then replace the residual exponential term with its original expansion from (\ref{eq4.9}). As a result we obtain the formal expansion
\begin{equation*}
\label{eq4.10}
\sum\limits_{s=1}^{\infty}
\frac{\iota_{s}s\mathrm{A}_{s}(\theta)}{u^{s+1}}
\sim
\sum\limits_{s=1}^{\infty}
\frac{\iota_{s}s\tilde{\mathbf{E}}_{s}(\nu,\theta)}
{u^{s+1}}
 \left\{1+ \sum\limits_{s=1}^{\infty}
\frac{\iota_{s}\mathrm{A}_{s}(\theta)}{u^{s}}\right\}.
\end{equation*}

We then use the Cauchy product of two infinite series \cite[Sect. 73]{Brown:2014:CVA}, and therefore we find that $\mathrm{A}_{1}(\theta)=\tilde{\mathbf{E}}_{1}(\nu,\theta)$, and for $s=2,3,4,\ldots$
\begin{equation}
\label{eq4.11}
\mathrm{A}_{s}(\theta) =
\tilde{\mathbf{E}}_{s}(\nu,\theta)
+\frac{1}{s}\sum_{j=1}^{s-1}\iota_{s,j}
j\tilde{\mathbf{E}}_{j}(\nu,\theta)
\mathrm{A}_{s-j}(\theta),
\end{equation}
where
\begin{equation}
\label{eq4.12a}
\iota_{s,j}=\frac{\iota_{j}\iota_{s-j}}{\iota_{s}}
=
    \begin{cases}
        -1 & \text{if } s \,\, \text{is even and}\,\,
        j \,\, \text{is odd}\\
        1 & \text{otherwise}
    \end{cases}.
\end{equation}

From (\ref{eq1.5}), (\ref{eq4.5}) and (\ref{eq4.9}) we then obtain our desired expansion
\begin{equation}
\label{eq4.12}
\hat{A}(u,\theta) \sim 
\sum\limits_{s=0}^{\infty}
\frac{A_{s}(\theta)}{u^{2s}},
\end{equation}
where $A_{0}(\theta)=1$, and  $A_{s}(\theta)=\mathrm{A}_{2s}(\theta)$ ($s=1,2,3,\ldots$). These coefficients are easily determined recursively using (\ref{eq2.14}) - (\ref{eq2.17}), (\ref{eq2.31}), (\ref{eq2.32}), (\ref{eq4.7}), (\ref{eq4.8}), (\ref{eq4.11}) and (\ref{eq4.12a}). For example, 
\begin{equation}
\label{eq4.13}
A_{1}(\theta)=
-\frac{\lambda \left( \lambda ^{2}-1 \right )}
{8\theta^2}
\left\{ (\lambda - 2)\theta^2\cot^{2}(\theta) - 
2\lambda\theta \cot(\theta) + \lambda + 2\right\},
\end{equation}
and
\begin{equation}
\label{eq4.14}
A_{2}(\theta)=\frac{\lambda 
\left( \lambda ^{2}-1 \right )(\lambda-2)}{384\,\theta^4}
\sum_{j=0}^{4} \alpha_{2,j}(\theta)\theta^j \cot^{j}(\theta),
\end{equation}
where
\begin{equation}
\label{eq4.15}
\alpha_{2,0}(\theta)
=(\lambda + 2)(\lambda + 4)(\lambda^2 - 9),
\end{equation}
\begin{equation}
\label{eq4.16}
\alpha_{2,1}(\theta)
=4 \lambda
\left\{6(\lambda + 1)\theta^2 - 
(\lambda - 1)(\lambda + 2)(\lambda + 3)\right\},
\end{equation}
\begin{equation}
\label{eq4.17}
\alpha_{2,2}(\theta)
=6\left\{4(4-\lambda^2 +\lambda)\theta^2 + 
\lambda(\lambda^2 - 1)(\lambda + 2)\right\},
\end{equation}
\begin{equation}
\label{eq4.18}
\alpha_{2,3}(\theta)
=-4\lambda(\lambda + 1)
(\lambda + 2)(\lambda - 3),
\end{equation}
and
\begin{equation}
\label{eq4.19}
\alpha_{2,4}(\theta)
=(\lambda + 2)(\lambda - 4)(\lambda^2 - 9).
\end{equation}

Following similar steps, if
\begin{equation}
\label{eq4.19a}
\exp \left\{
\sum\limits_{s=1}^{\infty}
\frac{\iota_{s}\mathbf{E}_{s}(\nu,\theta)}{u^{s}}
\right\}
\sim 1+\sum\limits_{s=1}^{\infty}
\frac{\iota_{s}\mathrm{B}_{s}(\theta)}{u^{s}},
\end{equation}
then we find that $\mathrm{B}_{1}(\theta)=\mathbf{E}_{1}(\nu,\theta)$ and for $s=2,3,4,\ldots$
\begin{equation}
\label{eq4.20}
\mathrm{B}_{s}(\theta) =
\mathbf{E}_{s}(\nu,\theta)
+\frac{1}{s}\sum_{j=1}^{s-1}\iota_{s,j}
j\mathbf{E}_{j}(\nu,\theta)
\mathrm{B}_{s-j}(\theta).
\end{equation}

Defining $B_{s}(\theta)=\mathrm{B}_{2s+1}(\theta)$ ($s=0,1,2,\ldots$), taking imaginary parts of both sides of (\ref{eq4.19a}) and comparing to (\ref{eq4.6}), we then obtain
\begin{equation}
\label{eq4.21}
\hat{B}(u,\theta) \sim 
\sum\limits_{s=0}^{\infty}
\frac{B_{s}(\theta)}{u^{2s+1}}.
\end{equation}

Again we can readily compute these terms recursively, and for the first two find that
\begin{equation}
\label{eq4.22}
B_{0}(\theta)=
\frac{\lambda (\lambda-1)}{2\theta}
\left\{ \theta\cot(\theta) - 1\right\},
\end{equation}
and 
\begin{equation}
\label{eq4.23}
B_{1}(\theta)=\frac{\lambda 
\left( \lambda ^{2}-1 \right )(\lambda-2)}{48\,\theta^3}
\sum_{j=0}^{3} \beta_{1,j}(\theta)\theta^j \cot^{j}(\theta),
\end{equation}
where
\begin{equation}
\label{eq4.24}
\beta_{1,0}(\theta)=-\beta_{1,3}(\theta)
=(\lambda + 2)(\lambda-3),
\end{equation}
\begin{equation}
\label{eq4.25}
\beta_{1,1}(\theta)
=3(2\theta^2-\lambda^2+\lambda),
\end{equation}
and
\begin{equation}
\label{eq4.26}
\beta_{1,2}(\theta)
=3\lambda(\lambda- 1).
\end{equation}

We next use the following theorem \cite[Thm. 3.1]{Dunster:2021:NKF} to establish that the coefficients in (\ref{eq4.12}) and (\ref{eq4.21}) are analytic at $\theta=0$ ($z=1$), and moreover these expansions are uniformly valid for $0 \leq \theta \leq \pi/2$.
\begin{theorem} 
\label{eqthm:nopoles}
Let $0<\rho_{1}<\rho_{2}$, $u>0$, $z_{0}\in \mathbb{C}$, $H(u,z)$ be an analytic function of $z$ in the open disk $D=\{z:\, |z-z_{0}|<\rho_{2}\}$, and $h_{s}(z)$ ($s=0,1,2,\ldots$) be a sequence of functions that are analytic in $D$ except possibly for an isolated singularity at $z=z_{0}$. If $H(u,z)$ is known to possess the asymptotic expansion
\begin{equation} 
\label{eq4.27}
H(u,z) \sim \sum\limits_{s=0}^{\infty}
\frac{h_{s}(z)}{u^{s}}
\quad (u \rightarrow \infty),
\end{equation}
in the annulus $\rho_{1}<|z-z_{0}|<\rho_{2}$, then $z_{0}$ is an ordinary point or a removable singularity for each $h_{s}(z)$, and the expansion (\ref{eq4.27}) actually holds for all $z \in D$ (with $h_{s}(z_{0})$ defined by the limit of this function at $z_{0}$ if it is a removable singularity).
\end{theorem}

This is applied to (\ref{eq3.35}) with
\begin{equation*}
\label{eq4.27a}
H(u,z)=
\frac{\Gamma(u)
\Gamma\left(\frac{1}{2}u - \frac{1}{2}\lambda + 1\right)}
{2^{\lambda - 1}
\Gamma\left(\frac{1}{2}u
+\frac{1}{2}\lambda\right)n!} A(u,z),
\end{equation*}
and the coefficients $h_{s}(z)$ given by the formal expansion
\begin{equation*}
\label{eq4.27b}
\exp \left\{
\sum\limits_{s=1}^{\infty}
\frac{\tilde{\mathbf{E}}_{2s}(\nu,z)}{u^{2s}}
\right\}
\cosh \left\{ \sum\limits_{s=0}^{\infty }
\frac{\tilde{\mathbf{E}}_{2s+1}(\nu,z)}{u^{2s+1}}
\right\}
\sim \sum\limits_{s=0}^{\infty}\frac{h_{s}(z)}{u^{s}}
\quad (u \rightarrow \infty).
\end{equation*}

All odd terms $h_{2s+1}(z)$ ($s=0,1,2,\ldots$) are zero in this notation, and the even terms $h_{2s}(z)$ ($s=0,1,2,\ldots$) are all seen to have an isolated singularity at $z=1$ by virtue of \cref{lemMero2}. We then have in the notation of the theorem $z_0=1$, $\rho_1>0$ arbitrarily small, and $\rho_2=1$, and so we deduce that the singularity of each coefficient is removable. By uniqueness of asymptotic expansions we see from (\ref{eq4.5}) and (\ref{eq4.12}) that $A_s(\theta)=h_{2s}(\cos(\theta))$ ($s=0,1,2,\ldots$), establishing that the expansion (\ref{eq4.12}) holds for $0 \leq \theta \leq \pi/2$, assuming of course that $A_{s}(0)=\lim_{\theta \rightarrow 0}A_{s}(\theta)$. Similarly we can prove the corresponding result for (\ref{eq4.21}).

To illustrate, we calculate from (\ref{eq4.13}) - (\ref{eq4.19}) as $\theta \rightarrow 0$ ($x = \cos(\theta) \rightarrow 1$)
\begin{equation}
\label{eq4.28}
A_{1}(\theta)=
-\tfrac{1}{6}\lambda\left(\lambda^2-1\right)
+\mathcal{O}\left(\theta^2\right),
\end{equation}
\begin{equation}
\label{eq4.29}
A_{2}(\theta)=
\tfrac{1}{360}\lambda(\lambda^2 - 1)
(\lambda - 2)(\lambda-3)(5\lambda+7)
+\mathcal{O}\left(\theta^2\right),
\end{equation}
and from (\ref{eq4.22}) - (\ref{eq4.26})
\begin{equation}
\label{eq4.30}
B_{0}(\theta)=
-\tfrac{1}{6}\lambda(\lambda - 1)\theta
+\mathcal{O}\left(\theta^3\right),
\end{equation}
\begin{equation}
\label{eq4.31}
B_{1}(\theta)=
-\tfrac{1}{120}\lambda(\lambda^2 - 1)
(\lambda - 2)\theta+\mathcal{O}\left(\theta^3\right).
\end{equation}

We remark that $A_s(\theta)$ and $B_s(\theta)$ are even and odd functions, respectively, which can be proven by induction using (\ref{eq4.11}) and (\ref{eq4.20}) and recalling that $A_{s}(\theta)=\mathrm{A}_{2s}(\theta)$ and $B_{s}(\theta)=\mathrm{B}_{2s+1}(\theta)$.

\subsection{Numerical results}

We would like to compute the relative error of the difference between the exact values to the approximate ones given in this section. However on account of the zeros of the polynomials it is preferable to work with an expression that does not vanish in the interval $0 \leq x = \cos(\theta) \leq 1$. A typical choice for a real-valued oscillatory differentiable function $f(x)$ with simple zeros is its envelope $\{f(x)^2+f'(x)^2\}^{1/2}$.

Now from \cite[Eq. 18.9.19]{NIST:DLMF}	
\begin{equation}
\label{eq4.32}
\frac{d}{dx}C_{n}^{(\lambda)}(x)
=C_{n-1}^{(\lambda+1)}(x),
\end{equation}
and hence this can be approximated for large $n$ by a simple modification of (\ref{eq4.4}). To obtain comparable sized magnitudes of oscillation, at least when the arguments are not too close to $x=1$, we refer to (\ref{eq3.24a}) to scale the derivative by the factor
\begin{equation*}
\label{eq4.33}
\left\{\frac{(\lambda)_{n} }
{2^{\lambda-(1/2)} n!}\right\}
\left\{\frac{2^{\lambda+(1/2)} (n-1)! }
{(\lambda+1)_{n-1}}\right\}
=\frac{2\lambda}{n}.
\end{equation*}

\begin{figure}[h!]
 \centering
 \includegraphics[
 width=0.7\textwidth,keepaspectratio]{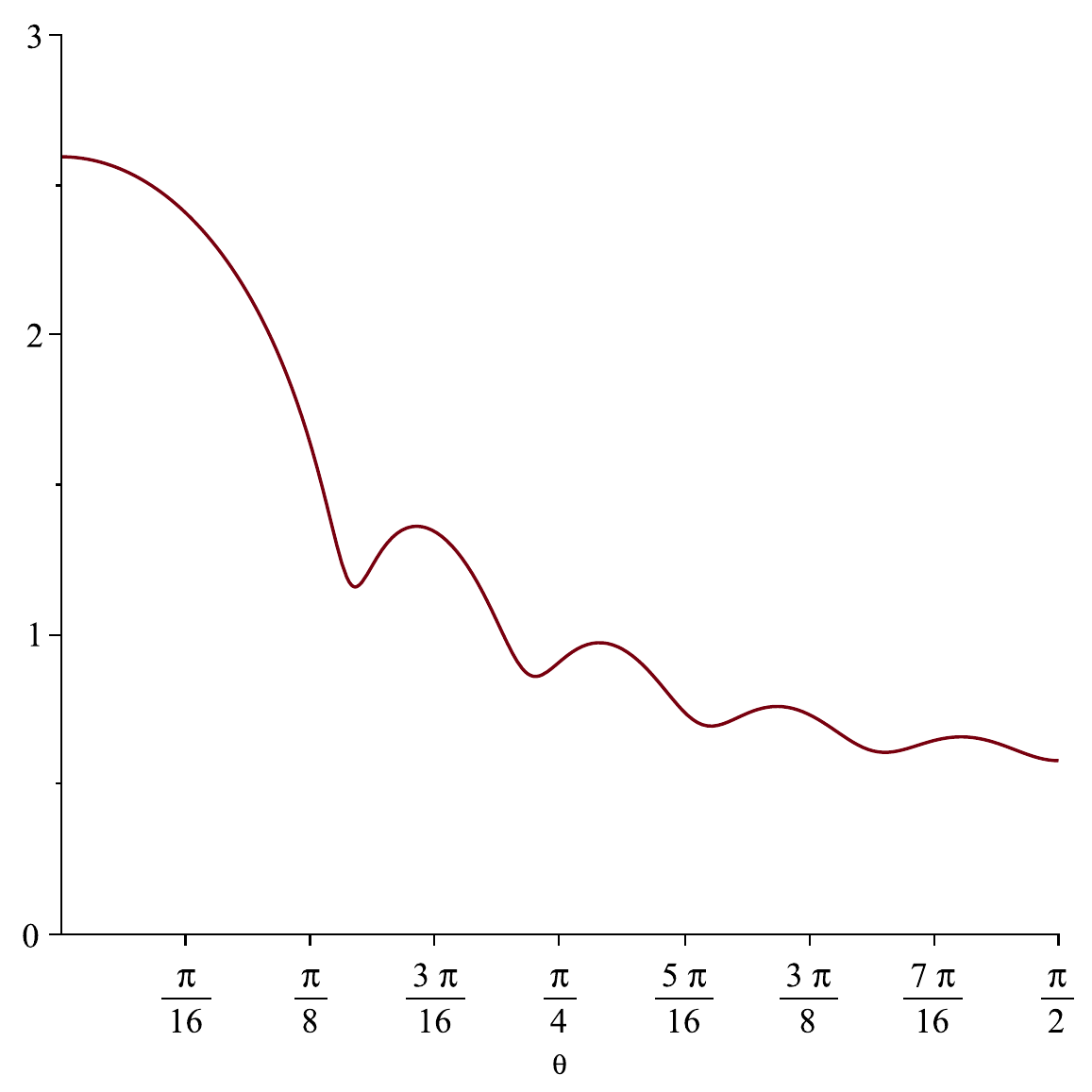}
 \caption{Graph of $\log_{10}(M_{n}^{(\lambda)}(\theta))$ with $n=10$ and $\lambda=1.7$ for $0\leq \theta \leq \pi/2$}
 \label{fig:fig1}
\end{figure}

Thus we define the envelope
\begin{equation}
\label{eq4.34}
M_{n}^{(\lambda)}(\theta)
=\left[\left\{C_{n}^{(\lambda)}(\cos(\theta))\right\}^2
+\left\{\frac{2\lambda}{n}
C_{n-1}^{(\lambda+1)}
(\cos(\theta))\right\}^2\right]^{1/2}.
\end{equation}

Let us choose $n=10$ and $\lambda=1.7$. \cref{fig:fig1} depicts a graph the logarithm (base 10) of $M_{10}^{(1.7)}(\theta)$ for $0\leq \theta \leq \pi/2$, illustrating that this function is fairly slowly varying and of course non-vanishing. We graph its logarithm since the maximum value $M_{10}^{(1.7)}(0)=392.308\cdots$ is quite large relative to the minimum value $M_{10}^{(1.7)}(\pi/2)=3.791\cdots$. Incidentally, this disparity in values is not an issue in the relative error evaluation that follows.

As an approximation to (\ref{eq4.4}) and (\ref{eq4.34}), respectively, we now define
\begin{multline}
\label{eq4.35}
\mathcal{C}_{n,N}^{(\lambda)}(\theta)
=\sqrt{\frac{u\pi}{2}}\,
\frac{ \Gamma\left(\frac{1}{2}u+\frac{1}{2}\lambda\right)
\theta^{1/2}}
{\Gamma\left(\frac{1}{2}u-\frac{1}{2}\lambda+1\right)
\Gamma(\lambda)
\{\sin(\theta) \}^{\lambda}}
\\ \times
\left\{J_{\nu}(u \theta) \sum\limits_{s=0}^{N}
\frac{A_{s}(\theta)}{u^{2s}}
-J_{\nu+1}(u \theta)
\sum\limits_{s=0}^{N-1}
\frac{B_{s}(\theta)}{u^{2s+1}} \right\},
\end{multline}
and
\begin{equation}
\label{eq4.36}
\mathcal{M}_{n,N}^{(\lambda)}(\theta)
=\left[\left\{\mathcal{C}_{n,N}^{(\lambda)}(\theta)
\right\}^2
+\left\{\frac{2\lambda}{n}
\mathcal{C}_{n-1,N}^{(\lambda+1)}(\theta)
\right\}^2\right]^{1/2}.
\end{equation}

The relative error we examine is then given by
\begin{equation}
\label{eq4.37}
\Delta_{n,N}^{(\lambda)}(\theta)
=\left\vert\frac{M_{n}^{(\lambda)}(\theta)
-\mathcal{M}_{n,N}^{(\lambda)}(\theta)}
{M_{n}^{(\lambda)}(\theta)}
\right\vert.
\end{equation}

\begin{figure}[h!]
 \centering
 \includegraphics[
 width=0.7\textwidth,keepaspectratio]{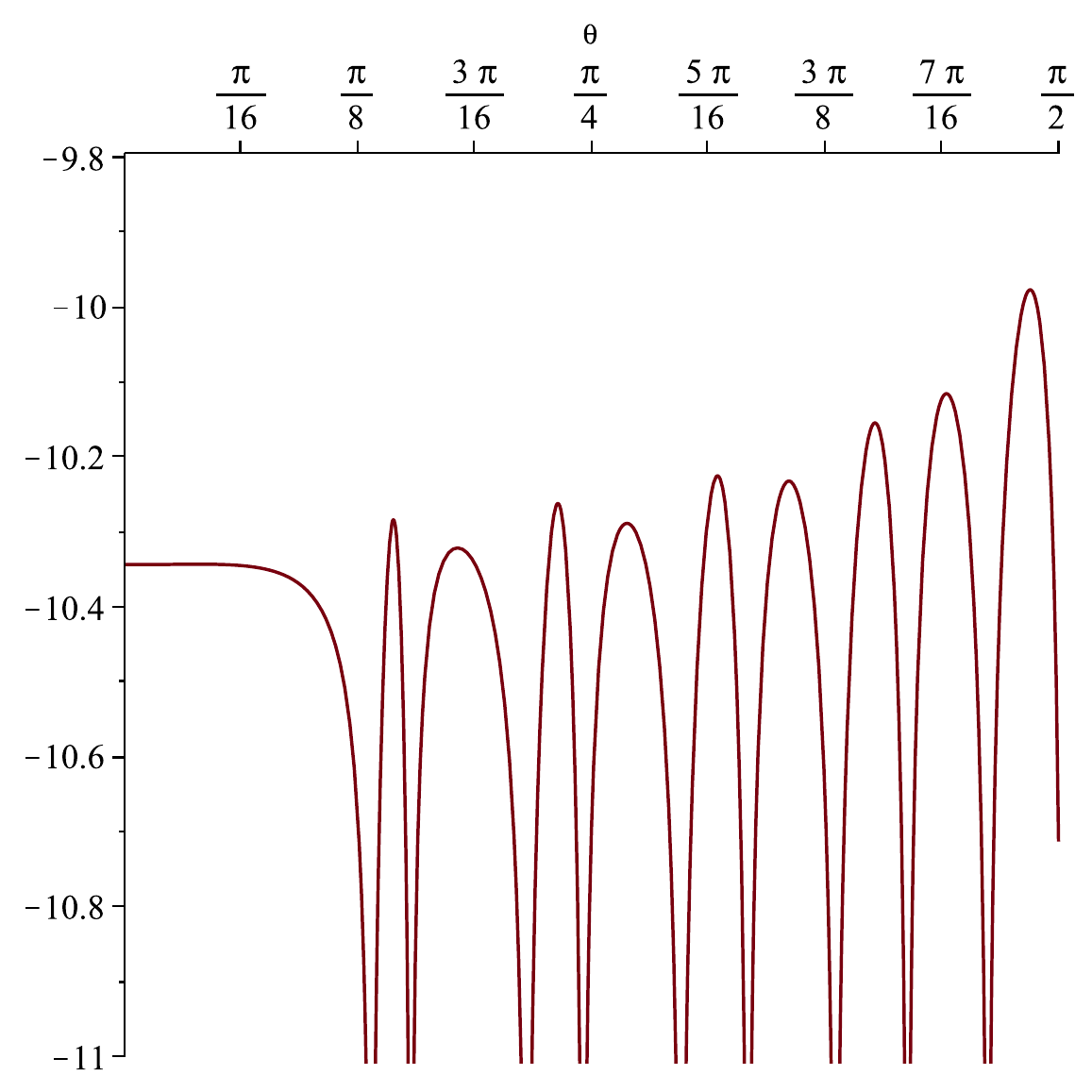}
 \caption{Graph of $\log_{10}(\Delta_{n,N}^{(\lambda)}(\theta))$ with $n=10$, $N=4$ and $\lambda=1.7$ for $0\leq \theta \leq \pi/2$}
 \label{fig:fig2}
\end{figure}

\begin{figure}[h!]
 \centering
 \includegraphics[
 width=0.7\textwidth,keepaspectratio]{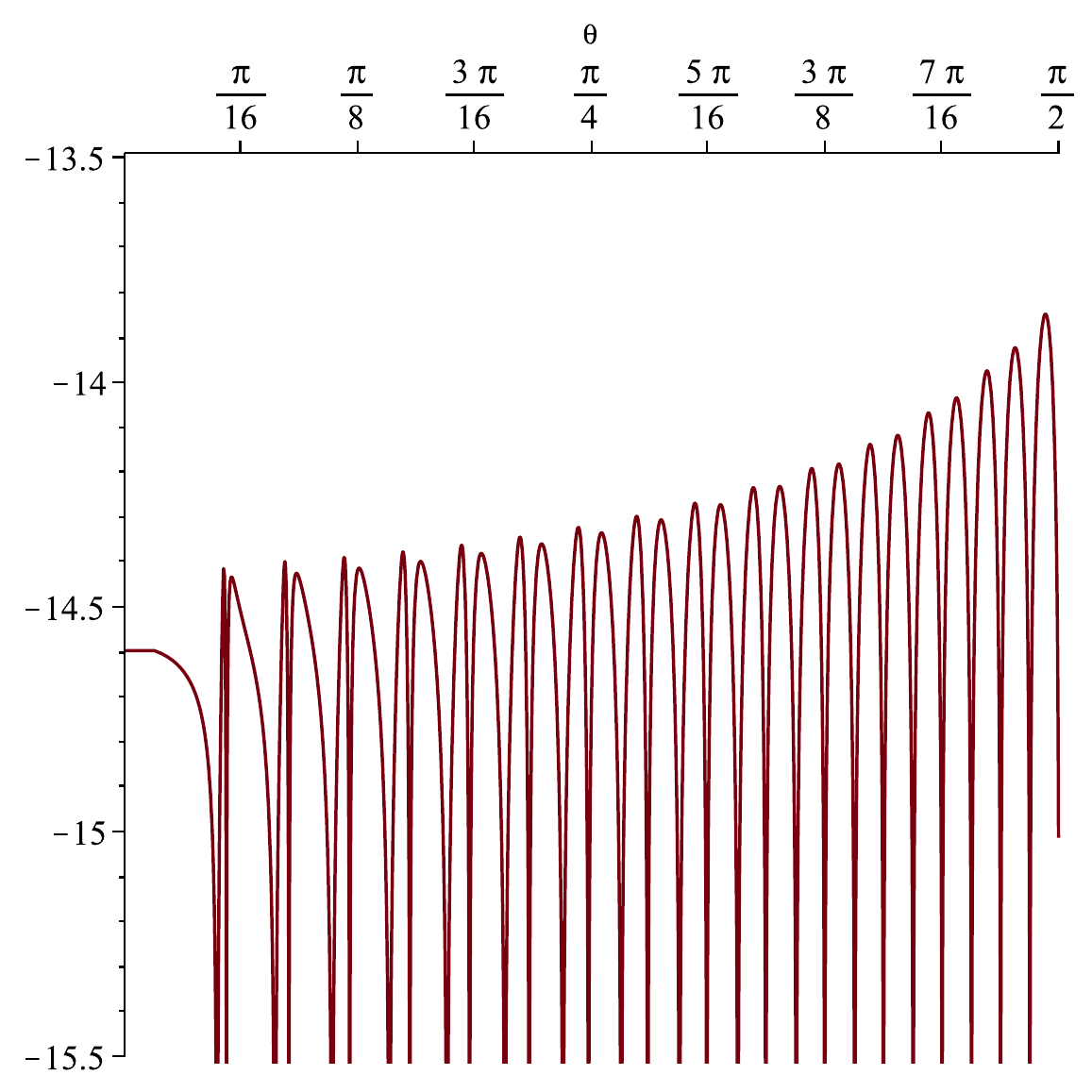}
 \caption{Graph of $\log_{10}(\Delta_{n,N}^{(\lambda)}(\theta))$ with $n=30$, $N=4$ and $\lambda=1.7$ for $0\leq \theta \leq \pi/2$}
 \label{fig:fig3}
\end{figure}

In \cref{fig:fig2} we plot the logarithm base 10 of this function with $n=10$ and $\lambda=1.7$ for $0\leq \theta \leq \pi/2$, taking $N=4$ in our approximation (\ref{eq4.35}) for the Gegenbauer polynomial. As expected, we observe about 10 digits of accuracy, uniformly throughout the interval. In \cref{fig:fig3} a similar plot is given for $n=30$ with the other parameters the same. This time we obtain about 14 digits of accuracy, uniformly throughout the interval. In both graphs the vertical asymptotes occur of course at points where the relative error is identically zero. These points of zero error are seen to be relatively equally distributed, which can presumably be explained by the same being true of the zeros of both the polynomials and their approximants.

\section{Error bounds}
\label{sec5}

We construct error bounds for the expansions of \cref{sec2}, and then use these to obtain error bounds for the expansions in \cref{sec3}. We focus on the first quadrant of the $z$ plane, since the important bounds in this section carry over to the fourth quadrant by Schwarz symmetry when applicable.

\subsection{Error bounds for LG expansions}

We first avoid the pole at $z=1$ by restricting $z$ to $0 \leq \arg(z) \leq \pi/2$ and $|z-1|\geq 1$, as shown by the unshaded region in \cref{fig:fig4}. We also will include the negative imaginary axis, as $z$ crosses the cut $[0,1]$ from above, in our error analysis.

Recalling (\ref{eq2.12}), the corresponding (bounded) region in the $\beta$ plane is illustrated in \cref{fig:fig5}, and in both figures corresponding points are labelled under this transformation. For example, $\mathsf{E}$ and $\mathsf{F}$ are points arbitrarily large in the $z$ plane on their respective axes, and these correspond to points lying on the real axis the $\beta$ plane that are vanishingly close to $\beta =1$. Similar descriptions apply to the points $\mathsf{B}$ and $\mathsf{C}$ with the roles of $z$ and $\beta$ reversed.

We denote $\Gamma$ as the semicircle $|z-1|=1$ in the first quadrant, shown as the boundary $\mathsf{A}\mathsf{G}\mathsf{D}$ in \cref{fig:fig4}, and we let $\Gamma^{\prime}$ be the curve in the $\beta$ plane corresponding to $\Gamma$, which is $\mathsf{A}\mathsf{G}\mathsf{D}$ in \cref{fig:fig5}

\begin{figure}[ht]
 \centering
 \includegraphics[
 width=0.7\textwidth,keepaspectratio]{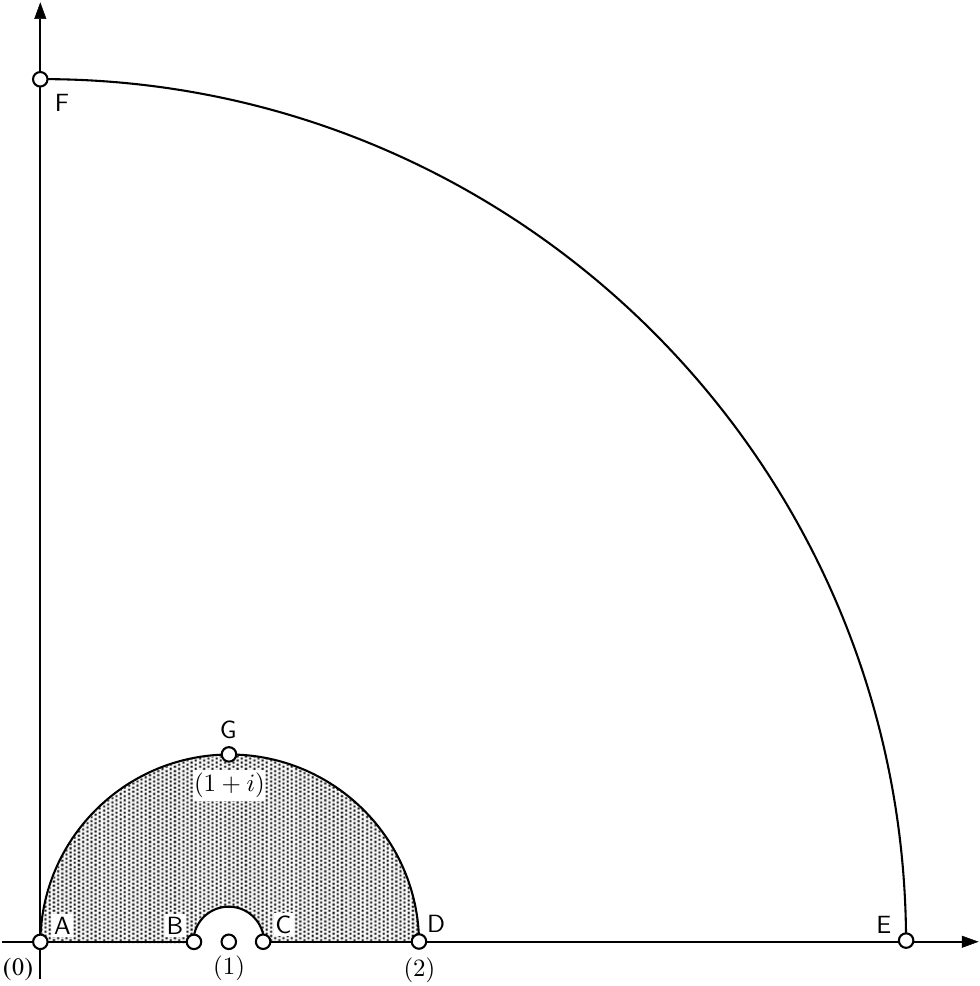}
 \caption{$z$ plane}
 \label{fig:fig4}
\end{figure}

\begin{figure}[ht]
 \centering
 \includegraphics[
 width=0.7\textwidth,keepaspectratio]{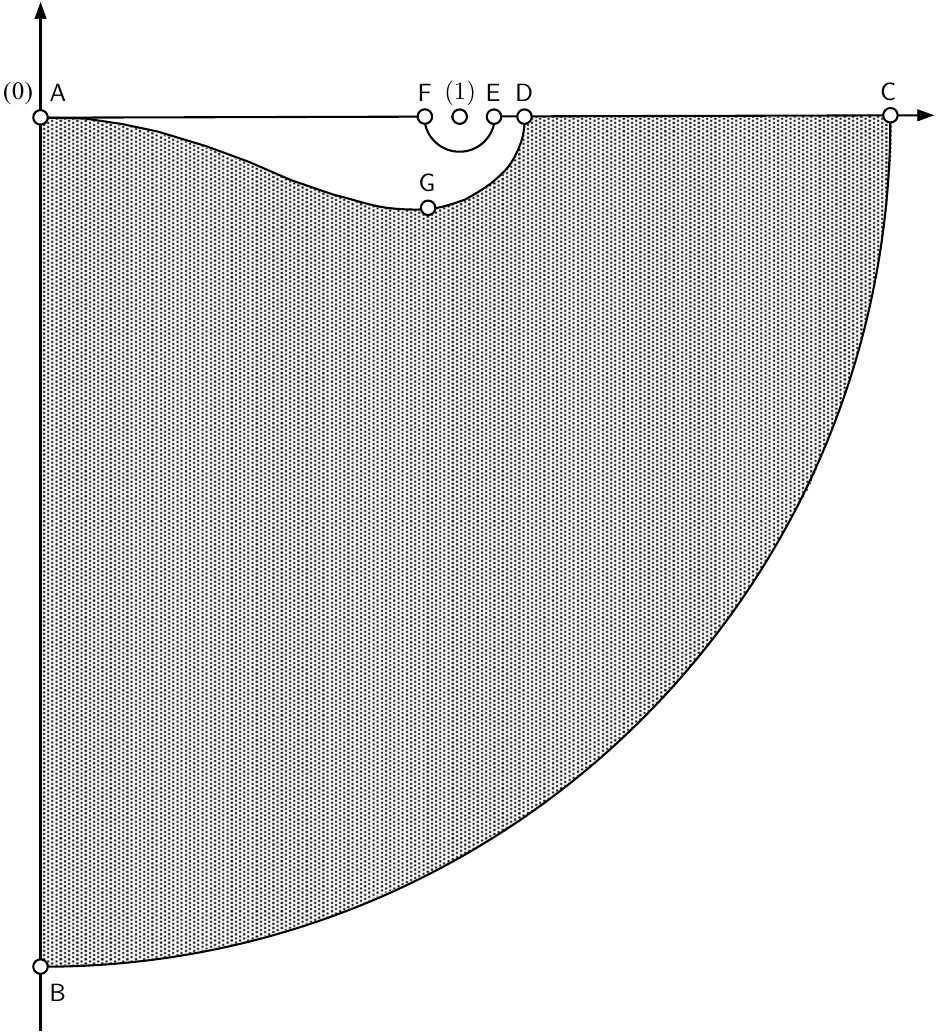}
 \caption{$\beta$ plane}
 \label{fig:fig5}
\end{figure}

Consider now the expansions (\ref{eq2.19}) and (\ref{eq2.21}) for the two solutions that are a numerically satisfactory pair in the region under consideration. We truncate these after $N-1$ terms, for arbitrary $N \geq 2$. Thus
\begin{equation}
\label{eq2.19z}
D_{n}^{(\lambda)}(z) = 
\frac{k_{n}(\nu)}
{\left\{4\left(z^2-1\right)\right\}
^{(2\nu+1)/4}}
\exp \left\{\sum\limits_{s=1}^{N-1}(-1)^{s+1}
\frac{\tilde{E}_{s}(1)}{u^{s}}
\right\}
W_{N,0}(u,z),
\end{equation}
and
\begin{equation}
\label{eq2.22z}
D_{n,-1}^{(\lambda)}(z) =
\frac{i k_{n}(\nu)}
{\left\{4\left(z^2-1\right)\right\}
^{(2\nu+1)/4}}
\exp \left\{\sum\limits_{s=1}^{N-1}(-1)^{s+1}
\frac{\tilde{E}_{s}(1)}{u^{s}}
\right\}
W_{N,-1}(u,z),
\end{equation}
where
\begin{equation}
\label{eq5.1}
W_{N,0}(u,z)=
\exp \left\{ -u\xi
+\sum\limits_{s=1}^{N-1}(-1)^{s}
\frac{\tilde{E}_{s}(\beta)}
{u^{s}}\right\}
\left\{ 1+\eta_{N,0}(u,\beta) \right\},
\end{equation}
and
\begin{equation}
\label{eq5.2}
W_{N,-1}(u,z)=
\exp \left\{u\xi
+\sum\limits_{s=1}^{N-1}
\frac{\tilde{E}_{s}(\beta)}
{u^{s}}\right\}
\left\{ 1+\eta_{N,-1}(u,\beta) \right\}.
\end{equation}

Then for $j=0,-1$ \cite[Eqs. (1.16) - (1.23)]{Dunster:2020:LGE} (with a slightly different notation for the functions and some parameters) provides the following error bounds
\begin{equation}
\label{eq5.3}
\left\vert \eta_{N,j}(u,\beta) \right\vert \leq u^{-N}\omega_{N,j}(u,\beta) 
\exp \left\{ u^{-1}\varpi_{N,j}(u,\beta) +
u^{-N}\omega_{N,j}(u,\beta) \right\}, 
\end{equation}
where, on referring to (\ref{eq2.12a}),
\begin{equation}
\label{eq5.4}
\omega_{N,j}(u,\beta) =2\int_{\beta^{(j)}}^{\beta}
{\left\vert
\frac{\tilde{F}_{N}(b) d b}{1-b^2}
\right\vert }
+\sum\limits_{s=1}^{N-1}\dfrac{1}{{u^{s}}}
\sum\limits_{k=s}^{N-1}
\int_{\beta^{(j)}}^{\beta}{\left\vert 
\frac{\tilde{F}_{k}(b) \tilde{F}_{s+N-k-1}(b) db}
{1-b^2}
\right\vert },
\end{equation}
and 
\begin{equation}
\label{eq5.5}
\varpi_{N,j}(u,\beta) =4\sum\limits_{s=0}^{N-2}
\frac{1}{u^{s}}
\int_{\beta^{(j)}}^{\beta}
\left\vert 
\frac{\tilde{F}_{s+1}(b) db}{1-b^2}
\right\vert.
\end{equation}

Here $\beta^{(0)}=1$ and $\beta^{(-1)}=-1$, corresponding to $z=+\infty$ and $z=-i\infty$, respectively, where in the latter case  $z \rightarrow - i \infty$ across cut $[-1,1]$. Recall $D_{n,-1}^{(\lambda)}(z) \rightarrow 0$ in this case. Note that all integrals in the bounds, as required, converge at $\beta^{(0)}$ and $\beta^{(-1)}$; see (\ref{eq2.14}) - (\ref{eq2.16}) and the sentence that proceeds them. The paths of integration are taken along $\tilde{\mathcal{L}}_{j}$, say, as described next.

From \cite[Thm. 1.1]{Dunster:2020:LGE} $\tilde{\mathcal{L}}_{j}$ must consist of a chain of $R_{2}$ arcs, and as the integration variable moves along this path from one end to the other $\Re(\xi)$ must be monotonic, where in our application $\xi$ is given by (\ref{eq2.13}). Let us describe the simplest choice of paths for the $\beta$ domain depicted in \cref{fig:fig5}, and in doing so demonstrate that all points in the domain can be accessed by such a path, with guarantees that the bounds (\ref{eq5.3}) - (\ref{eq5.5}) hold there (and equivalently $z$ lying in the domain shown in \cref{fig:fig4}).

Firstly, we observe from (\ref{eq2.13}) that $\Re(\xi)$ is monotonic along a path iff $|(\beta+1)/(\beta-1)|$ is monotonic along that path. A straightforward calculation shows that the level curves $|(\beta+1)/(\beta-1)|=\text{constant}$ lying in the fourth quadrant of the $\beta$ plane are the family of semicircles satisfying
\begin{equation*}
\label{eq5.6}
\left\{\Re(\beta)-c\right\}^2+\Im(\beta)^2=c^2-1,
\end{equation*}
where the $c>1$ is any constant. Thus they are centered on the real axis at $\beta=c>1$ with radius $\sqrt{c^2-1}$. Hence they all intercept the positive real axis at two points, namely $\beta =c-\sqrt{c^2-1}$ which is to the left of $\beta=1$, and the other at $\beta=c+\sqrt{c^2-1}$ which is to the right of $\beta=1$.

It appears then that path $\tilde{\mathcal{L}}_{0}$ can simply be taken as the straight line from $\beta=1$ to the point in question, since from the level curves it seems evident that $\Re(\xi)$ is monotonic along such a path. This is readily verified by writing $\beta = x + iy$ ($x,y \in \mathbb{R}$) where $y=m(x-1)$ ($m \neq 0$ being the slope of the line), and in this case
\begin{equation*}
\label{eq5.7}
\left|\frac{\beta+1}{\beta-1}\right|^2=
\frac { \left(m^{2}+1 \right)x^{2}
-2\left(m^{2}-1\right) x
+m^{2}+1}{ \left( x-1 \right) ^{2} 
\left( m^{2}+1\right) }.
\end{equation*}
Hence
\begin{equation}
\label{eq5.8}
\frac{d}{dx}\left|\frac{\beta+1}{\beta-1}\right|^2=
4\frac {x+1}{ (1-x)^{3} \left(m^{2}+1 \right) }.
\end{equation}

Thus if $m>0$ we have $x>1$ and $x$ increasing as $\beta$ goes along $\tilde{\mathcal{L}}_{0}$ from $1$ to the end point, and (\ref{eq5.8}) implies that $|(\beta+1)/(\beta-1)|$ is decreasing, as asserted. Similarly,  if $m<0$ we have $x<1$ and $x$ decreasing as $\beta$ goes along $\tilde{\mathcal{L}}_{0}$ from $1$ to the end point, and (\ref{eq5.8}) again implies that $|(\beta+1)/(\beta-1)|$ is decreasing. 

For the vertical line $x=1$, $0<y<\infty$, the monotonicity requirement immediately follows from
\begin{equation*}
\label{eq5.9}
\left|\frac{\beta+1}{\beta-1}\right|=
\sqrt{1+\frac{4}{y^2}}.
\end{equation*}

Consider now $\tilde{\mathcal{L}}_{-1}$ joining $\beta^{(-1)} =-1 $ to any point in the fourth quadrant of the $\beta$ plane. This path will include the line segment $-1 \leq \beta \leq 0$. If $\beta \in [0,1]$ then $\tilde{\mathcal{L}}_{-1}$ can simply be taken as the real line segment $[-1,\beta]$. Otherwise, our choice of a typical path $\tilde{\mathcal{L}}_{-1}$ is shown in \cref{fig:fig6}, joining $\beta=-1$ with a point in the fourth quadrant. We take it to be the union of a line segment along the real axis and part of the (semicircular) level curve in the fourth quadrant that passes through the point in question. As noted above, all such level curves intersect the positive real axis at a point to the left of $\beta =1$, so such a path exists for all points in the fourth quadrant, and in particular all points in the region under consideration depicted in \cref{fig:fig5}. Note that this includes points lying on the real axis in the interval $(1,\infty)$ since the family of semicircular level curves also all intersect the real axis at a point to the right of $\beta =1$.

\begin{figure}[ht]
 \centering
 \includegraphics[
 width=0.7\textwidth,keepaspectratio]{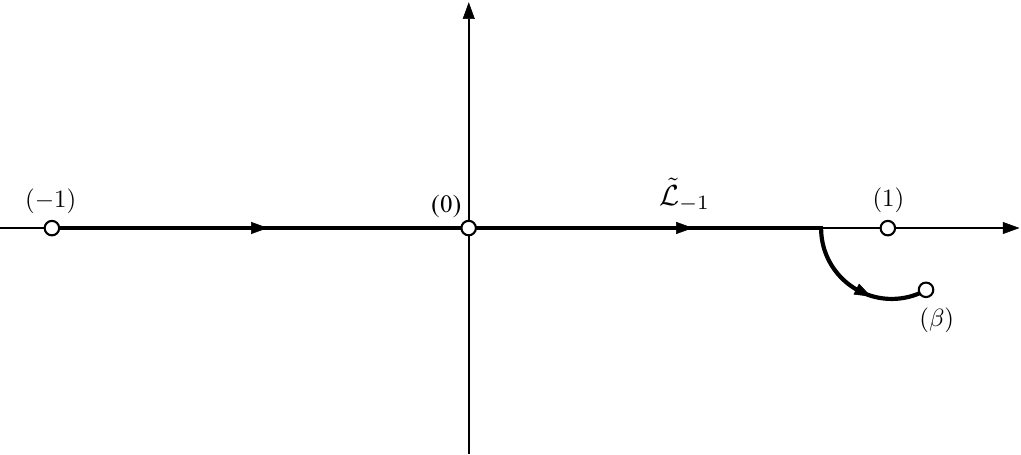}
 \caption{Path $\tilde{\mathcal{L}}_{-1}$ in $\beta$ plane}
 \label{fig:fig6}
\end{figure}

Monotonicity of $|(\beta+1)/(\beta-1)|$ along the level curve part of $\tilde{\mathcal{L}}_{-1}$ is evident by how they are defined, since this quantity is constant on that arc. Along the line segment part monotonicity is also clearly demonstrated by the simple calculation
\begin{equation*}
\label{eq5.10}
\frac{d}{d\beta}\left(\frac{\beta+1}{1-\beta}\right)=
\frac {2}{\left(1-\beta\right)^2 } >0
\quad  (-1\leq \beta<1).
\end{equation*}

Consider next error bounds for the expansions (\ref{eq2.29}) and (\ref{eq2.32a}) for the modified Bessel functions. As $z \rightarrow -i \infty$ across cut $[-1,1]$ then $\sqrt{z^2 - 1} \rightarrow i \infty$, and hence $\xi \sim -\ln(2|z|)+\frac{1}{2}\pi i$. Call this point $\xi^{(-1)}$. In particular $\xi \rightarrow \infty$ above the cut along negative real axis, and hence $K_{\nu}(u\xi e^{-\pi i})$ is the modified Bessel function which is recessive in this case. Similarly $K_{\nu}(u\xi e^{\pi i})$ is recessive for $z \rightarrow i \infty$ across cut $[-1,1]$, with $\xi \sim -\ln(2|z|)-\frac{1}{2}\pi i$ and this point at infinity is labeled $\xi^{(1)}$

Again using \cite{Dunster:2020:LGE}, error bounds are constructed in a similar manner to above, and we arrive at
\begin{equation} 
\label{eq5.11}
K_{\nu}(u\xi) =
\left(\frac{\pi}{2u\xi}\right)^{1/2}
\exp\left\{-u\xi-\sum_{s=1}^{N-1} 
(-1)^s\frac{a_{s}(\nu)}{s (u\xi)^{s}}
\right\}\left\{ 1+\eta_{N,0}^{(K)}(\nu,u,\xi) \right\},
\end{equation}
and
\begin{equation} 
\label{eq5.12}
K_{\nu}\left(u\xi e^{\pm \pi i}\right) =
\mp i \left(\frac{\pi}{2u\xi}\right)^{1/2}
\exp\left\{u\xi-\sum_{s=1}^{N-1} 
\frac{a_{s}(\nu)}{s (u\xi)^{s}}
\right\}\left\{ 1+\eta_{N, \pm 1}^{(K)}(\nu,u,\xi) \right\},
\end{equation}
where for $j=0,\pm 1$
\begin{equation}
\label{eq5.13}
\left\vert \eta_{N,j}^{(K)}(\nu,u,\xi) \right\vert \leq u^{-N}\omega_{N,j}^{(K)}(\nu,u,\xi) 
\exp \left\{ u^{-1}\varpi_{N,j}^{(K)}(\nu,u,\xi) +
u^{-N}\omega_{N,j}^{(K)}(\nu,u,\xi) \right\}, 
\end{equation}
in which 
\begin{equation}
\label{eq5.14}
\omega_{N,j}^{(K)}(\nu,u,\xi) =2\int_{\xi^{(j)}}^{\xi}
{\left\vert
\frac{a_{N}(\nu) d t}{t^{N+1}}
\right\vert }
+\sum\limits_{s=1}^{N-1}\dfrac{1}{{u^{s}}}
\sum\limits_{k=s}^{N-1}
\int_{\xi^{(j)}}^{\xi}{\left\vert 
\frac{a_{k}(\nu)a_{s+N-k-1}(\nu) dt}
{t^{s+N+1}}
\right\vert },
\end{equation}
and 
\begin{equation}
\label{eq5.15}
\varpi_{N,j}^{(K)}(\nu,u,\xi) =4\sum\limits_{s=0}^{N-2}
\frac{1}{u^{s}}
\int_{\xi^{(j)}}^{\xi}
\left\vert 
\frac{a_{s+1}(\nu) dt}{t^{s+2}}
\right\vert.
\end{equation}

Here $\xi^{(0)}=+\infty$ and $\xi^{(\pm 1)}=-\infty \mp \frac{1}{2}\pi i$ as defined above, and the paths are taken so that $\Re(t)$ is monotonic as $t$ moves along the path from $t=\xi^{(j)}$ ($j=0,\pm 1$) to $t=\xi$, and avoids the singularity $\xi = 0$ (which corresponds to $z=1$). The collection of all such points defines the region of validity of (\ref{eq5.11}) - (\ref{eq5.15}) for each $j$.

We only require the two values $j=0,-1$, since these correspond to $K_{\nu}(u\xi)$ and $K_{\nu}(u\xi e^{-\pi i})$, which are the numerically satisfactory pair in the first quadrant of the $z$ plane. It is straightforward to show that the intersection of the regions of validity for their bounds corresponds to a region in the $z$ plane which certainly contains the part of the first quadrant $0 \leq \arg(z) \leq \pi/2$, $|z-1| \geq 1$ shown in \cref{fig:fig4}, as required.

\subsection{Error bounds and computation of the expansions valid at the pole $z=1$}

Here we obtain error bounds for the truncated versions of (\ref{eq3.32}) and (\ref{eq3.33}). The advantage of these over obtaining one bound, such as Olver did for (\ref{eq3.6b}) is twofold. Firstly, the coefficient functions are slowly varying, and as a result the bounds for each of them are much simpler and do not have to mimic the more complicated behavior of the approximants, namely the modified Bessel functions. Secondly, these coefficient functions are the same for all solutions, so different error bounds are not required for these other approximations.

For simplicity we assume $N \geq 3$ is odd, but this can easily be modified in our analysis if $N$ is chosen to be even. From (\ref{eq3.29}), (\ref{eq3.29a}), (\ref{eq2.19z}) - (\ref{eq5.2}), (\ref{eq5.11}) and (\ref{eq5.12}) we then write
\begin{equation}
\label{eq5.16}
A(u,z) = 
\frac{\Gamma(u+\lambda) n!}
{\Gamma(u)\Gamma(u+1)}
\exp \left\{
\sum\limits_{s=1}^{N-1}(-1)^{s+1}
\frac{\tilde{E}_{s}(1)}{u^{s}}
\right\}\left\{A_N(u,z)+\epsilon_{N}^{(A)}(u,z)\right\},
\end{equation}
where
\begin{equation}
\label{eq5.17}
A_N(u,z) = 
\exp \left\{
\sum\limits_{s=1}^{(N-1)/2}
\frac{\tilde{\mathcal{E}}_{2s}(\nu,z)}{u^{2s}}
\right\}
\cosh \left\{ \sum\limits_{s=0}^{(N-3)/2}
\frac{\tilde{\mathcal{E}}_{2s+1}(\nu,z)}{u^{2s+1}}
\right\},
\end{equation}
and
\begin{multline}
\label{eq5.18}
\epsilon_{N}^{(A)}(u,z) =
\frac{1}{2} \exp \left\{\sum\limits_{s=1}^{N-1} 
\frac{\tilde{\mathcal{E}}_{s}(\nu,z)}{u^{s}}
\right\}
\\ \times
\left\{
\eta_{N,0}(u,\beta) +\eta_{N,-1}^{(K)}(\nu+1,u,\xi)
+\eta_{N,0}(u,\beta) \eta_{N,-1}^{(K)}(\nu+1,u,\xi)
\right\}
\\ 
+\frac{1}{2} \exp \left\{
\sum\limits_{s=1}^{N-1} 
(-1)^s \frac{\tilde{\mathcal{E}}_{s}(\nu,z)}{u^{s}}
\right\}
\\ \times
\left\{\eta_{N,-1}(u,\beta) +\eta_{N,0}^{(K)}(\nu+1,u,\xi)
+\eta_{N,-1}(u,\beta) \eta_{N,0}^{(K)}(\nu+1,u,\xi)
\right\}.
\end{multline}

Similarly, using (\ref{eq3.29a}) instead of (\ref{eq3.29}), we write
\begin{equation}
\label{eq5.16a}
B(u,z) = 
\frac{\Gamma(u+\lambda) n!}
{\Gamma(u)\Gamma(u+1)}
\exp \left\{
\sum\limits_{s=1}^{N-1}(-1)^{s+1}
\frac{\tilde{E}_{s}(1)}{u^{s}}
\right\}\left\{B_N(u,z)+\epsilon_{N}^{(B)}(u,z)\right\},
\end{equation}
where
\begin{equation}
\label{eq5.17a}
B_N(u,z) = 
\exp \left\{
\sum\limits_{s=1}^{(N-1)/2}
\frac{\mathcal{E}_{2s}(\nu,z)}{u^{2s}}
\right\}
\sinh \left\{ \sum\limits_{s=0}^{(N-3)/2}
\frac{\mathcal{E}_{2s+1}(\nu,z)}{u^{2s+1}}
\right\},
\end{equation}
and
\begin{multline}
\label{eq5.18a}
\epsilon_{N}^{(B)}(u,z) =
\frac{1}{2} \exp \left\{\sum\limits_{s=1}^{N-1} 
\frac{\mathcal{E}_{s}(\nu,z)}{u^{s}}
\right\}
\\ \times
\left\{
\eta_{N,0}(u,\beta) +\eta_{N,-1}^{(K)}(\nu,u,\xi)
+\eta_{N,0}(u,\beta) \eta_{N,-1}^{(K)}(\nu,u,\xi)
\right\}
\\ 
-\frac{1}{2} \exp \left\{
\sum\limits_{s=1}^{N-1} 
(-1)^s \frac{\mathcal{E}_{s}(\nu,z)}{u^{s}}
\right\}
\\ \times
\left\{\eta_{N,-1}(u,\beta) +\eta_{N,0}^{(K)}(\nu,u,\xi)
+\eta_{N,-1}(u,\beta) \eta_{N,0}^{(K)}(\nu,u,\xi)
\right\}.
\end{multline}

Bounds for the error terms $\epsilon_{N}^{(A)}(u,z)$ and $\epsilon_{N}^{(B)}(u,z)$ that are valid away from the pole $z=1$, and certainly for $|z-1| \geq 1$ in the first quadrant, can be immediately be derived by inserting the LG error bounds (\ref{eq5.3}) - (\ref{eq5.5}), (\ref{eq5.13}) - (\ref{eq5.15}) into (\ref{eq5.18}) and (\ref{eq5.18a}). These bounds immediately carry over to the fourth quadrant by Schwarz symmetry, establishing that the error terms are uniformly $\mathcal{O}(u^{-N})$ as $u \rightarrow \infty$.

In order to get error bounds that are valid for $|z-1|<1$, as well as computing our expansions there, we use Cauchy's integral formula, bearing in mind that $A(u,z)$ and $B(u,z)$ are analytic in this disk. To this end we obtain from (\ref{eq5.16})
\begin{multline}
\label{eq5.19}
A(u,z) = 
\frac{1}{2 \pi i} \oint_{|t-1|=1} \frac{A(u,t) dt}{t-z}
=\frac{\Gamma(u+\lambda) n!}
{\Gamma(u)\Gamma(u+1)}
\exp \left\{
\sum\limits_{s=1}^{N-1}(-1)^{s+1}
\frac{\tilde{E}_{s}(1)}{u^{s}}
\right\}
\\ \times
\left\{\frac{1}{2 \pi i} \oint_{|t-1|=1} 
\frac{A_N(u,t) dt}{t-z} 
+\delta_{N}^{(A)}(u,z)\right\},
\end{multline}
where
\begin{equation}
\label{eq5.20}
\delta_{N}^{(A)}(u,z)
=\frac{1}{2 \pi i} \oint_{|t-1|=1} \frac{
\epsilon_{N}^{(A)}(u,t) dt}{t-z}.
\end{equation}

Note that neither $A_N(u,z)$ nor $\epsilon_{N}^{(A)}(u,z)$ are analytic at $z=1$, but when combined they are, and so in deriving (\ref{eq5.19}) and (\ref{eq5.20}) we have split the original Cauchy integral into the separate ones involving these two meromorphic functions. We cannot then of course infer from (\ref{eq5.20}) that $\delta_{N}^{(A)}(u,z)$ and $\epsilon_{N}^{(A)}(u,z)$ are equal, indeed the former is analytic at $z=1$ (and as noted above the latter is not).

Let us now use (\ref{eq5.20}) to bound this error term for $|z-1|<1$. To do so we use \cite[Lemma 4.1]{Dunster:2021:SEB} with specific values for our disk:
\begin{lemma}
\label{lemK}
For $|z-1|< 1$
\begin{equation}
\label{eqK1}
\oint_{|t-1|=1}\left\vert {\dfrac{dt}{t-z}}
\right\vert =l(z) :=\frac{4 K(k) }{|z-1| +1},
\end{equation}
where
\begin{equation}
\label{defk}
k=\frac{{2}\sqrt{|z-1|}}{|z-1| +1},
\end{equation}
and $K(k)$ is the complete elliptic integral of the first kind defined by (\cite[\S 19.2(ii)]{NIST:DLMF})
\begin{equation}
\label{Kelliptic}
{K}(k) =\int\limits_{0}^{\pi /2}{\dfrac{d\phi }
{\sqrt{1-k^{2}\sin ^{2}(\phi) }}}=\int\limits_{0}^{1}
\dfrac{dt}{\sqrt{\left( 1-t^{2}\right) 
\left( 1-k^{2}t^{2}\right) }}
\quad (0\leq k<1).
\end{equation}
\end{lemma}

Thus from (\ref{eq5.20}) and (\ref{eqK1})
\begin{equation}
\label{eq5.20a}
\left|\delta_{N}^{(A)}(u,z)\right|
\leq \frac{1}{2 \pi } \sup_{z \in \Gamma}
\left|\epsilon_{N}^{(A)}(u,z)\right| l(z),
\end{equation}
where we recall that $\Gamma$ is the semicircle $|z-1|=1$ lying in the first quadrant. We only need to consider the supremum over this half of the boundary of the disk since $\overline{\epsilon_{N}^{(A)}(u,\bar{z})}=\epsilon_{N}^{(A)}(u,z)$.

We now majorize all terms on the RHS of (\ref{eq5.18}) for $z \in \Gamma$, or equivalently $\beta \in \Gamma'$ and $\xi \in \Gamma''$, where $\Gamma''$ is the contour in the $\xi$ plane corresponding to $\Gamma$. This is shown in \cref{fig:fig7}, and lies in the first quadrant of the $\xi$ plane with end points $\xi = \ln(2+\sqrt{3})$ and $\xi = \pi i/2$.

\begin{figure}[ht]
 \centering
 \includegraphics[
 width=0.7\textwidth,keepaspectratio]{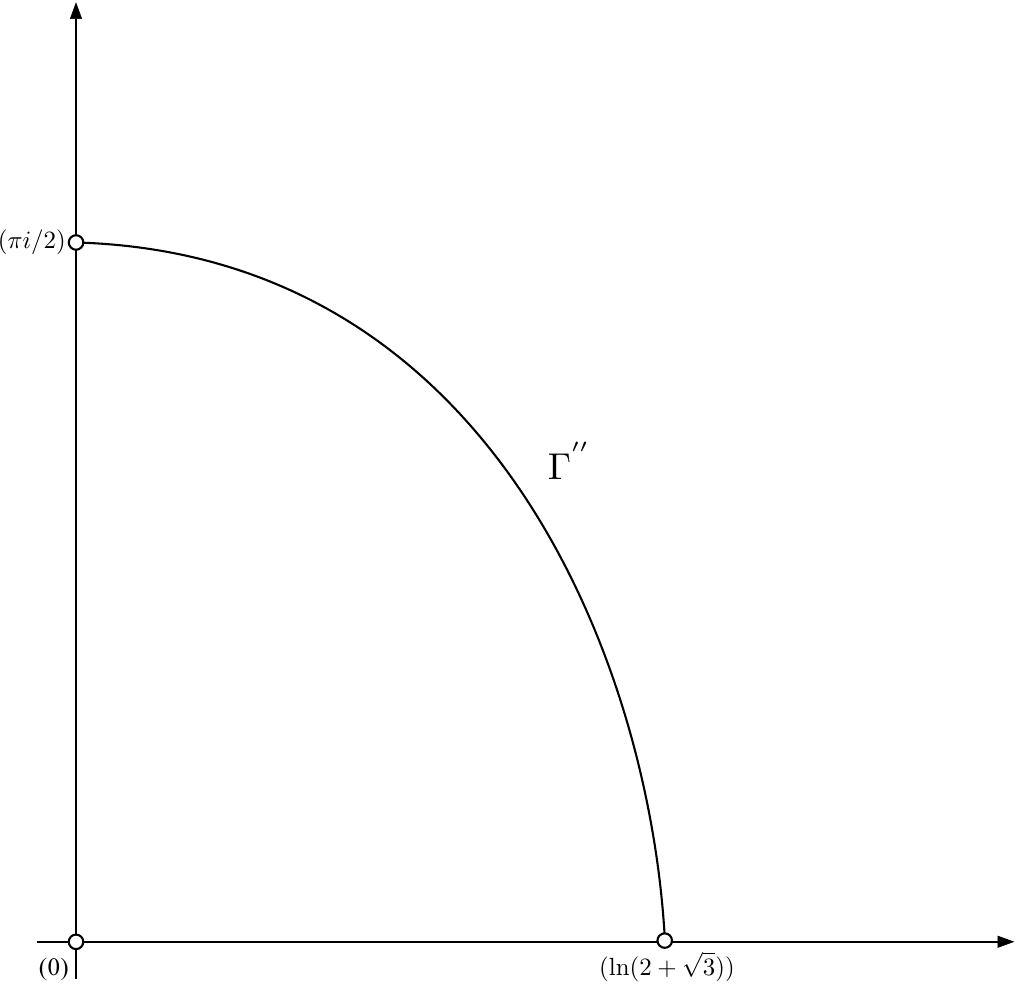}
 \caption{Path $\Gamma''$ in $\xi$ plane}
 \label{fig:fig7}
\end{figure}

Thus, let $M_{N}(\nu,u)=\max\{M_{N}^{\pm}(\nu,u)\}$ where
\begin{equation}
\label{eq5.21}
M_{N}^{\pm}(\nu,u)=\sup_{z \in \Gamma}\left\vert
\exp \left\{
\sum\limits_{s=1}^{N-1} 
(\pm 1)^s \frac{\mathcal{E}_{s}(\nu,z)}{u^{s}}
\right\} \right\vert,
\end{equation}
and $\tilde{M}_{N}(\nu,u)$ similarly defined but with $\mathcal{E}$ replaced by $\tilde{\mathcal{E}}$.

Also define $\eta_{N}(u)=\max\{\eta_{N,0}(u),\eta_{N,-1}(u)\}$ where
\begin{equation}
\label{eq5.22}
\eta_{N,j}(u)
=\sup_{\beta \in \Gamma^{\prime}}\left\vert
\eta_{N,j}(u,\beta) \right\vert
\quad (j=0,-1),
\end{equation}
and similarly let $\eta_{N}^{(K)}(\nu,u)=\max\{\eta_{N,0}^{(K)}(\nu,u),\eta_{N,-1}^{(K)}(\nu,u)\}$, where
\begin{equation}
\label{eq5.23}
\eta_{N,j}^{(K)}(\nu,u)
=\sup_{\xi \in \Gamma''}\left\vert
\eta_{N,j}^{(K)}(\nu,u,\xi) \right\vert
\quad (j=0,-1).
\end{equation}

Note $M_{N}(\nu,u),\tilde{M}_{N}(\nu,u) \rightarrow 1$ and $\eta_{N}(u),\eta_{N}^{(K)}(\nu,u)=\mathcal{O}(u^{-N})$ as $u \rightarrow \infty$, with the latter two being explicitly bounded via (\ref{eq5.3}) - (\ref{eq5.5}) and (\ref{eq5.13}) - (\ref{eq5.15}).

From (\ref{eq5.18}), (\ref{eq5.20a}) and (\ref{eq5.21}) - (\ref{eq5.23}) we arrive at our desired bound
\begin{multline}
\label{eq5.24}
\left\vert \delta_{N}^{(A)}(u,z) \right\vert
\leq \frac{1}{2 \pi } 
\tilde{M}_{N}(\nu,u)
\\ \times
\left\{\eta_{N}(u) +\eta_{N}^{(K)}(\nu+1,u)
+\eta_{N}(u) \eta_{N}^{(K)}(\nu+1,u)
\right\} l(z)
\quad (|z-1|<1).
\end{multline}

Similarly from (\ref{eq5.16a}) and (\ref{eq5.18a})
\begin{multline}
\label{eq5.25}
B(u,z) = 
\frac{\Gamma(u+\lambda) n!}
{\Gamma(u)\Gamma(u+1)}
\exp \left\{
\sum\limits_{s=1}^{N-1}(-1)^{s+1}
\frac{\tilde{E}_{s}(1)}{u^{s}}
\right\}
\\ \times
\left\{\frac{1}{2 \pi i} \oint_{|t-1|=1} 
\frac{B_N(u,t) dt}{t-z} 
+\delta_{N}^{(B)}(u,z)\right\},
\end{multline}
where $B_N(u,z)$ is given by (\ref{eq5.17a}), with the error bound
\begin{multline}
\label{eq5.27}
\left\vert \delta_{N}^{(B)}(u,z) \right\vert
\leq \frac{1}{2 \pi } 
M_{N}(\nu,u)
\\ \times
\left\{\eta_{N}(u) +\eta_{N}^{(K)}(\nu,u)
+\eta_{N}(u) \eta_{N}^{(K)}(\nu,u)
\right\} l(z)
\quad (|z-1|<1).
\end{multline}

\section*{Acknowledgments}
I thank the anonymous referee for helpful comments.

Support was provided from project PID2021-127252NB-I00 funded by \newline MCIN/AEI/10.13039/501100011033/ FEDER, UE.

\bibliographystyle{siamplain}
\bibliography{biblio}

\end{document}